\documentclass[a4paper]{amsart}
\usepackage{psfrag,graphicx,
}
\usepackage{amsfonts}
\usepackage{
amsmath,amscd,amsthm,amssymb,amsxtra,latexsym,epsf, 
epic,eepic,mathrsfs}
\usepackage{epsf,latexsym}
\usepackage[all]{xy}
\newfont{\Bb}{msbm10 scaled1200}
\newcommand{\basering}[1]{\ensuremath{\mathbb{#1}}}
\newcommand{\ZZ}{\basering{Z}}

\newcommand{\CC}{\basering{C}}

\newcommand{\NN}{\basering{N}}

\newcommand{\QQ}{\basering{Q}}
\DeclareMathOperator{\OO}{\mathcal O}

\begin{document}
\def\fitt{\text Fitt}
\def\vd{\text{VD}_\infty}
\def\sig{{\tilde\Sigma}}
\def\alt{{\text{\scriptsize Alt}}}
\def\pmat{\begin{pmatrix}}
\def\emat{\end{pmatrix}}
\def\trace{\text{trace }}
\def\cl{\centerline}
\def\ig{\includegraphics}
\def\vs{\vskip 10pt}
\def\vvs{\vskip 5pt}
\def\sign{\text{sign }}
\def\vsn{\vskip 10pt\ni}
\def\cc{\circ}
\def\y{\item}
\def\ni{\noindent}
\def\ben{\begin{enumerate}}
\def\een{\end{enumerate}}
\def\beq{\begin{equation}}
\def\ssm{\smallsetminus}
\def\eeq{\end{equation}}
\def\bit{\begin{itemize}}
\def\eit{\end{itemize}}
\def\ec{\end{center}}
\def\bc{\begin{center}}
\def\ld{{\ldots}}
\def\cd{{\cdots}}


\def\d{\mathbf{d}}
\def\e{\mathbf{e}}
\def\Gl{\mbox{Gl}}
\def\L{\mathfrak{l}}
\def\pgl{\mathfrak{pgl}}
\def\gl{\mathfrak{gl}}
\def\gg{\mathfrak{g}}
\def\la{\langle}
\def\ra{\rangle}


\def\th{{\theta}}
\def\ve{{\varepsilon}}
\def\pp{{\varphi}}
\def\a{{\alpha}}
\def\l{{\lambda}}
\def\rr{{\rho}}
\def\ss{{\sigma}}
\def\g{{\gamma}}
\def\G{{\Gamma}}
\def\o{{\omega}}
\def\O{{\Omega}}
\def\k{{\kappa}}


\def\Proof{{\bf Proof}\hspace{.2in}}
\def\eop{\hspace*{\fill}$\Box$ \vskip \baselineskip}
\def\eoq{{\hfill{$\Box$}}}

\def\bu{\bullet}
\def\tr{{\pitchfork}}
\def\b{{\bullet}}
\def\w{{\wedge}}
\def\p{{\partial}}

\def\fitt{\text{Fitt}}
\def\supp{\mbox{supp}}
\def\im{\mbox{im}}
\def\Der{\mbox{Der}}
\def\coker{\text{Coker }}
\def\ker{\text{Ker }}
\def\Mat{\mbox{Mat}}
\def\Rep{\mbox{Rep}}
\def\sgn{\mbox{sign}}
\def\cod{\mbox{codim }\ }
\def\hh{\mbox{height }\ }
\def\opp{\mbox{\scriptsize opp}}
\def\vol{\mbox{vol}}
\def\deg{\mbox{deg}}
\def\Sym{\mbox{Sym}}
\def\Ext{\mbox{Ext}}
\def\Hom{\mbox{Hom}}
\def\End{\mbox{End}}
\def\Aut{\mbox{Aut}}
\def\DD{{\Der(\log D)}}
\def\Dh{{\Der(\log h)}}
\def\pd{\mbox{pd}}
\def\dim{{\text{dim}\,}}
\def\depth{{\mbox{depth}}}


\def\EE{{\cal E}}
\def\R{{\cal R}}
\def\L{{\cal L}}
\def\C{{\cal C}}
\def\A{{\cal A}}
\def\TT{{\cal T}}
\newcommand{\bbbz}{{\mathbb Z}}


\newcommand{\n}[1]{\| #1 \|}
\newcommand{\um}[1]{{\underline{#1}}}
\newcommand{\om}[1]{{\overline{#1}}}
\newcommand{\fl}[1]{\lfloor #1 \rfloor}
\newcommand{\ce}[1]{\lceil #1 \rceil}
\newcommand{\ncm}[2]
{{\left(\!\!\!\begin{array}{c}#1\\#2\end{array}\!\!\!\right)}}
\newcommand{\ncmf}[2]
{{\left[\!\!\!\begin{array}{c}#1\\#2\end{array}\!\!\!\right]}}
\newcommand{\ncms}[2]
{{\left\{\!\!\!\begin{array}{c}#1\\#2\end{array}\!\!\!\right\}}}

\def\ig{\includegraphics}
\def\iff{{\Leftrightarrow}}
\def\imp{{\Rightarrow}}
\def\to{{\ \rightarrow\ }}
\def\too{{\ \longrightarrow\ }}
\def\into{{\hookrightarrow}}
\def\st{\stackrel}
\def\rank{\text{rank\,}}
\def\mm{\mathfrak{m}}

\newtheorem{theorem}{Theorem}[section]
\newtheorem{lemma}[theorem]{Lemma}
\newtheorem{sit}[theorem]{}
\newtheorem{lemmadefinition}[theorem]{Lemma and Definition}
\newtheorem{proposition}[theorem]{Proposition}
\newtheorem{example}[theorem]{Example}
\newtheorem{question}[theorem]{Question}
\newtheorem{remark}[theorem]{Remark}
\numberwithin{equation}{section}
\newtheorem{corollary}[theorem]{Corollary}
\newtheorem{definition}[theorem]{Definition}
\newtheorem{conjecture}[theorem]{Conjecture}
\title{Disentanglements of corank 2 map-germs: two examples 
}
\author{David Mond}
\address{Mathematics Institute, University of Warwick, Coventry CV4 7AL}
\email{d.m.q.mond@warwick.ac.uk}

\date{\today}
\subjclass{14B05, 32S30, 32S25 }
\begin{abstract}We compute the homology of the multiple point spaces of stable perturbations of two germs $(\CC^n,0)\to (\CC^{n+1},0)$ of corank 2, using a variety of techniques based on the image computing spectral sequence ICSS. We provide a reasonably detailed introduction to the ICSS, including some low-dimensional examples of its use. The paper is partly expository.
\end{abstract}
\maketitle
\section{Introduction}
In studying a singularity of mapping from $n$-space to $(n+1)$-space, a r\^ole analogous to that of Milnor fibre is played by a stable perturbation of the singularity, and in particular by its image. The image of a map acquires non-trivial homology through the identification of points of the domain, and these identifications are encoded in the multiple point spaces of the map. For germs of corank 1, these multiple point spaces are well understood. For germs of corank $>1$ the situation is radically different. 

In this paper we study the multiple point spaces of stable perturbations of two map-germs of corank 2 from $n$-space to $(n+1)$-space. In one case $n=3$ and in the other $n=5$.  Previous work of Marar, Nu\~no-Ballesteros and Pe\~nafort, in \cite{mn},\cite{mnp} has explored the case where $n=2$. Increasing the dimension introduces new difficulties. Confronting these will require a range of new techniques. Our work here is a preliminary exploration. Following the invitation of the editors to provide an accessible account, we have expanded the preliminary material on multiple-point spaces, disentanglements, and the image computing spectral sequence ICSS,  our principle technical tool, and included, in Section \ref{Examples}, some examples of calculation using the ICSS. 

The first of the two corank 2 map-germs we look at is the  germ  of lowest codimension in Sharland's list (\cite{sharland14}) of weighted homogeneous corank 2 map-germs $(\CC^3,0)\to (\CC^4,0)$: 
\beq\label{deff}f_0:(\CC^3,0)\to (\CC^4,0),\quad  f_0(x,y,z)=(x,y^2+xz+x^2y,yz,z^2+y^3).\eeq
This has ${\mathcal A}_e$-codimension 18. The second is the simplest ${\mathcal A}_e$-codimension 1 corank 2 map germ,
\beq\label{526}f_0:(\CC^5,0)\to (\CC^6,0), \quad
f_0(x,y,a,b,c)=(x^2+ax+by,xy,y^2+cx+ay,a,b,c).\eeq

For each one, we calculate a number of (topological) homology groups with rational coefficients, related to its disentanglement.  By ``disentanglement'' we do not mean just the stable perturbation 
$$\xymatrix{f_t:U_t\ar@{->>}[r] &X_t}\subset\CC^{n+1}$$
of the germ $f_0$ (where $U_t$ is a contractible neighbourhood of $0$ in $\CC^n$), as the term has been used by de Jong and van Straten in \cite{dJvS} and by Houston in \cite{Hou97} and subsequent papers.  A richer picture is obtained by considering the ``semi-simplicial resolution''
\beq\label{ssr}
\xymatrix{\ldots \ar@<5pt>[d]\ar@<2pt>[d]\ar@<-2pt>[d]\ar@<-5pt>[d]\ar[dddrr]\\
D^{3}(f_t)\ar@<3pt>[d]\ar@<0pt>[d]\ar@<-3pt>[d]\ar[ddrr]\\
D^2(f_t)\ar@<2pt>[d]\ar@<-2pt>[d]\ar[drr]\\
U_t\ar@{->>}[rr]&&X_t
}
\eeq
Here, for each integer $2\leq k\leq n$, $D^k(f_t)$ is the closure, in $U_t^k$, of the set of $k$-tuples of pairwise distinct points (``strict'' k-tuple points) sharing the same image, and the $k$ distinct arrows $\pi^k_j:D^k(f_t)\to D^{k-1}(f_t)$, $1\leq j\leq k$, are the restriction of the projections 
$U_t^k\to U_t^{k-1}$ obtained by forgetting the $j$'th
factor in the product $U_t^k$. For $k>n$, an $\mathcal A$-finite mono-germ $f:(\CC^n,0)\to (\CC^{n+1},0)$ will have no strict $k$-tuple points, since the dimension of $D^k(f_0)$ at a strict $k$-tuple point is $n-k+1$ (see Subsection \ref{mpi} below). In this case $D^k(f_0)$ is defined by a slightly different procedure: we pick a stable unfolding $F:(\CC^n\times\CC^d,0)\to (\CC^{n+1}\times\CC^d,0)$ of $f_0$, define $D^k(F)$ as above, and then take $D^k(f_0)$ as the fibre over $0\in\CC^d$ of $D^k(F)$. We note that it is an easy consequence of the Mather-Gaffney criterion for
$\mathcal A$-finiteness that if we apply this second procedure when $k\leq n$, we get  the same space $D^k(f_0)$ as defined above. 


The disentanglement, in this wider sense, contains complete information about the way that points of $U_t$ are identified by $f_t$. The image $X_t$ has the homotopy type of a wedge of  $n$-spheres  (\cite{siersma}) whose number, the ``image Milnor number" of $f$, $\mu_I(f)$, is the key geometric invariant of an $\mathcal A$- finite germ $(\CC^n,0)\to (\CC^{n+1},0)$. Since the homology of $X_t$ arises through the identifications induced by $f_t$, it is better described by the information 
attached to the diagram \eqref{ssr}. This will become clearer in what follows. 

Note that the $\pi_j^k$ for fixed $k$ and different $j$ are left-right equivalent to one another thanks to the symmetric group actions on $D^k$ and $D^{k-1}$, permuting the copies of $U_t$. In what follows we will consider only $\pi^k_k$, which we will refer to simply as $\pi^k$. We will denote the image of $\pi^k$ in $D^{k-1}$ by $D^k_{k-1}$, and, more
generally, for $\ell<k$, we denote the image of $\pi^{\ell+1}\circ\cd\circ\pi^k$ in $D^\ell$ by $D^k_\ell$. 
\vsn

\begin{remark}\label{inter}{\em 
(1) 
Any finite map-germ $f:(\CC^n,S)\to (\CC^{n+1},0)$ is an embedding outside $D^2_1(f)$, which is the ``non-embedding locus'' of $f$. More generally each map $\pi^k:D^k(f)\to D^{k-1}(f)$ is an embedding outside $D^{k+1}_k(f)$, and each map $\pi^{k+1}$ parameterises the non-embedding locus of its successor $\pi^k$. Thus the tower \eqref{ssr} shows a strong analogy with a free resolution of a module. If $f$ is stable then 
$D^k(f)$, if not empty, is $n-k+1$-dimensional. It follows that the length of this resolution is at most $n$.
\vskip 10pt \ni
(2) For maps $M^n\to N^{n+1}$ with $n<6$ there  is no stable singularity of corank 2. Every  $\mathcal A$-finite germ is stable outside $0$, so if $n<6$, any singularity outside $0$ of an $\mathcal A$-finite germ $f_0:(\CC^n,0)\to (\CC^{n+1},0)$, must be of corank $1$.  For stable germs of corank 1, all non-empty multiple point spaces are smooth (\cite{marar-mond}). It follows that for any $\mathcal A$-finite germ $f_0:(\CC^n,0)\to (\CC^{n+1},0)$ with $n<6$, $D^k(f_0)$ has (at most) isolated singularity. It also follows that a stable perturbation $f_t$ has no singularities of corank $>1$. Therefore all of 
the non-empty multiple point spaces $D^k(f_t)$ are smooth  -- indeed, are smoothings of the  isolated singularities $D^k(f_0)$. For any map $f$, $D^\ell(\pi^k)$ can be identified with $D^{k+\ell-1}(f)$, by the obvious map 
$$\left((x_1,\ld, x_{k-1}, x_k^{(1)}),(x_1,\ld, x_{k-1}, x_k^{(2)}), \ld, (x_1,\ld, x_{k-1}, x_k^{(\ell)})\right)$$
\beq\label{poi}\longleftrightarrow \left(x_1,\ld,x_k^{(1)},x_k^{(2)},\ld, x_k^{(\ell)}\right)\eeq
 -- the left hand side here shows a point of $D^\ell(\pi^k)$, and the right hand side shows the corresponding point of $D^{k+\ell-1}(f)$. This observation is the basis of the ``method of iteration''  developed by Kleiman in \cite{kmulti1}. From the smoothness of the $D^{k+j}(f_t)$ therefore follows smoothness of the multiple-point spaces of the projections $\pi^k:D^k(f_t)\to D^{k-1}(f_t)$. The singularities of $\pi^k$ are all of corank 1; this can be seen quite easily by writing $f$ in linearly adapted coordinates, but see also \cite{am}. By the characterisation of the stability of corank 1 map-germs by the smoothness of their multiple point spaces (\cite{marar-mond}), it follows that provided $n<6$, {\it all of the projections $\pi^k$ are stable maps}. }
\end{remark}
{\it Acknowledgements} 
The calculations in Section \ref{3to4} were begun in collaboration with Isaac Bird, as part of his final year MMath project at Warwick. I am grateful to him for agreement to use them here, and for his enthusiasm on the project, which contributed significantly to its further development. I also thank Mike Stillman for help finding the $2\times 4$ matrices of Section \ref{shadow}. 
\subsection{Multiple points}\label{mpi}
If $f_0:(\CC^n,0)\to (\CC^{n+1},0)$ is $\mathcal A$-finite then the set of strict $k$-tuple points is dense in $D^k(f_0)$, unless $D^k(f_0)$ consists only of the point $(0,\ld,0)$. If $(x_1,\ld, x_k)$ is a strict $k$-tuple point, with
$f_0(x_i)=y$ for $i=1,\ld, k$, 
then by the Mather-Gaffney criterion for $\mathcal A$-finiteness, the images of the germs $f_0:(\CC^n,x_i)\to (\CC^{n+1},y)$ meet in general position. It follows that their intersection has dimension $n+1-k$. This is therefore
the dimension of $D^k(f_0)$, provided $k\leq n+1$, and of $D^k(f_t)$. If $k>n+1$ then because $f_t$ is stable,
$D^k(f_t)=\emptyset$.
\subsection{Alternating homology}\label{althom}
The developments in this section are due principally (but in some cases implicitly) to Goryunov in \cite{gor}.
\vsn
{\it Notation} For a continuous map $\pp:V\to W$, we denote by $\pp_\#$ the map $C_j(V)\to C_j(W)$
induced by $\pp$, and reserve the term $\pp_*$ for the corresponding map on homology.
\vsn
Suppose $f:X\to Y$ is surjective. Recall the action of $S_k$ on $D^k(f)$, permuting the copies of $X$. With $C_\bu(D^k(f))$ the usual singular chain complex, 
define
$$C_j^\alt(D^k(f))=\{c\in C_j(D^k(f)):\sigma_\#(c)=\text{sign}(\sigma) c\ \text{ for all }\sigma\in S_k\}.$$
This gives a subcomplex, as 
$\p_\#(C_j^\alt)\subset C^\alt_{j-1}, $
so we have {\it alternating homology}
$$H^\alt_j(D^k(f)).$$

\ni Now observe that also $\pi^k_\#:C_j^\alt(D^k(f))\subset C_j^\alt(D^{k-1}(f))$. To see this, let $\sigma\in S_{k-1}$, and define $\tilde\sigma\in S_k$ by setting $\tilde\sigma(i)=\sigma(i)$ for $1\leq i\leq k-1$ and $\tilde\sigma (k)=k$. Then $\text{sign}(\tilde\sigma)=\text{sign}(\sigma)$, and so if $c\in C^\alt_j(D^k(f))$ then
$$\sigma_\#(\pi^k_\#(c))=\pi^k_\#(\tilde\sigma_\#(c))=\pi^k_\#(\text{sign}(\tilde\sigma)c)=
\text{sign}(\sigma)\pi^k_\#(c).$$

In fact we have a double complex:  on $C_j^\alt(D^k(f))$, 
$\pi^{k-1}_\#\circ\pi^k_\#=0;$
 for 
$$\pi^{k-1}_\#\circ\pi^k_\#=\pi^{k-1}_\#\circ\pi^k_\#\circ(k,k-1)_\#,$$ 
and on alternating chains $(k,k-1)_\#$ is multiplication by $-1$. 
By same argument, $f_\#\circ \pi^2_\#=0.$ Thus, denoting $X$ by $D^1(f)$, $Y$ by $D^{-1}(f)$, and $f$ by $\pi^1$, we have
\begin{proposition}$(C_j^\alt(D^\bu(f)),\pi^\bu)$ is a complex, and 
$(C^\alt_\bu(D^\bu(f)), \p, (-1)^\bu\pi^\bu_\#)$ is a double complex. \eop
\end{proposition}
The relevance to the homology of the image can be seen from two short calculations. 
\vskip 10pt\ni Example 1:  let $c_j^2\in Z_j^\alt(D^2(f))$.
$$\xymatrix{D^2(f)&&0&c^2_j\ar[l]_{\p}\ar[d]\\
X&&0&\pi^2_{\#}(c^2_j)\ar[l]\ar[d]^{f_\#}&c^1_{j+1}\ar[l]^{\exists}_\p\ar[d]&&\text{e.g. if } H_j(X)=0\\
Y&&&f_\#\pi^2_\#(c^2_j)=0&f_\#(c^1_{j+1}){\ar[l]_\p}&}
$$
Because $f_\#\circ\pi^2_\#=0$ on alternating chains, $f_\#(c^1_{j+1})$ is a cycle in $Y$.  So from an alternating $j$-cycle $c^2_j$ in 
$D^2(f)$, we get a $j+1$ cycle on $Y$ -- provided $\pi^2_\#(c^2_j)$ is a boundary in $X$, i.e. provided $\pi^2_*[c^2_j]=0$ in $H_j(X)$.
\vskip 10pt
\ni Example 2: let $c_j^3\in Z_j(D^3(f))$. 
$$\xymatrix{
D^3(f)&0&c_j^3\ar[l]_\p\ar[d]\\
D^2(f)&0&\pi^3_\#(c^3_j)\ar[l]_\p\ar[d]&c^2_{j+1}\ar[l]_\p^{\exists}\ar[d]&&\text{provided } \pi^3_*[c^3_j]=0\in H^\alt_j(D^2(f))\\
X&&0&\pi^2_\#(c^2_{j+1})\ar[l]_\p\ar[d]&c^1_{j+2}\ar[l]^{\exists}_\p\ar[d]&\text{provided } \pi^2_*[c^2_{j+1}]=0\in H_{j+1}(X)\\
Y&&&0&f_\#(c^1_{j+2})\ar[l]_\p
}
$$
Here, a $j$-dimensional homology class in $D^3(f)$ leads to a $j+2$-dimensional class in $Y$, provided certain homology classes vanish. 
 
Note that in both cases, \begin{itemize}
\y If $c^k_j=\pi^{k+1}_\#(c^{k+1}_j)$ for $c^k_j\in C^\alt_j(D^{k+1}(f))$  then $\pi^k_\#(c^k_j)=0$. 
\y If $c^k_j=\p c^k_{j+1}$ for some $c^k_{j+1}\in C^\alt_{j+1}(D^k(f))$, then we can take $c_{j+1}^{k-1}=\pi^k_\#(c^k_{j+1})$
so the homology class we get in $H^\alt_{j+1}(D^{k-1}(f))$ is zero.
\end{itemize} 
So we are really interested in
$$\frac{\ker \pi^k_*:H^\alt_j(D^k(f))\to H_j^\alt(D^{k-1}(f))\quad \quad}{\im\, \pi^{k+1}_*:H^\alt_j(D^{k+1}(f))\to H_j^\alt(D^k(f))}$$
\subsection{The image-computing spectral sequence}
Lurking behind the two calculations we have just gone through is the 
{\it Image-computing spectral sequence}, ICSS. This was introduced in \cite{gm} and further developed in \cite{gor}. It calculates the homology of the image $X_t$ in terms of the alternating homology $H_*^\alt(D^k(f_t))$ of the multiple point spaces 
$D^k(f_t)$.  The version introduced in \cite{gm} worked with the subspace of $H_*(D^k(f);\QQ)$ on which 
$S_k$ acts by its sign representation:
$$\text{Alt\,} H_i(D^k(f;\QQ)=\{[c]\in H_i(D^k(f);\QQ):\sigma_*([c])=\text{sign}(\sigma)[c]\ \text{for all }\sigma\in S_k\}.$$ 
If we take the complex of alternating chains described in the last paragraph and replace integer coefficients by rational coefficients, then the two versions coincide:
$$\text{Alt\,}H_j(D^k(f);\QQ)=H^\alt_j(D^k(f);\QQ).$$  

The ICSS has $E^1_{p,q}=H^\alt_{q}(D^{p+1}(f_t))$ and converges to $H_{p+q}(X_t)$. The differential on the $E^1$ page, $d^1:E^1_{p,q}\to E^1_{p-1,q}$ is the simplicial differential $\pi_*^{p+1}:H_q^\alt(D^{p+1}(f_t))\to H_q^\alt(D^p(f_t))$. In \cite{gm}, a great deal hinges on the fact that for a stable perturbation $f_t$ of an $\mathcal A$-finite germ $f_0$ of corank 1, the $D^k(f_t)$ are Milnor fibres of the isolated complete intersection singularities $D^k(f_0)$ (see \cite{marar-mond}), and therefore their vanishing homology is confined to middle dimension. Since (over $\QQ$) $H_i^\alt(D^k(f_t))\subset H_i(D^k(f_t))$, the vanishing alternating homology of $D^k(f_t)$ is also confined to middle dimension. From this it follows, in the case of a stable perturbation of a mono-germ,  that the ICSS collapses at $E^1$: for all $r\geq 1$, $E^r_{p,q}=E^1_{p,q}$. The fact that the spectral sequence converges to $H_{p+q}(X_t)$ therefore means that, as $\QQ$-vector space, 
\beq\label{icsscor}H_n(X_t)\simeq H_{n-1}^\alt(D^2(f_t))\oplus H_{n-2}^\alt(D^3(f_t))\oplus\cd\oplus H_0^\alt(D^{n+1}(f_t)).\eeq 
The argument for collapse is as follows: for each space $D^{p+1}(f_t)$ there is at most one non-zero alternating homology group, $H_{n-p}^\alt(D^{p+1}(f_t))$, and therefore either the source or the target of every differential at $E^1$ is equal to $0$.  Thus $E^2_{p,q} = E^1_{p,q}$. The higher differentials $d^r:E^r_{p,q}\to E^r_{p-r,q+r-1}$ all vanish for exactly the same reason: for each one, either its source or its target is zero. 

Notice that this is exactly what is needed to justify the assumptions we made in our two calculations in the previous paragraph. Whenever $D^k(f_t)$ has non-trivial alternating homology in dimension $j$, then
$D^{k-1}(f)$ does not.

The situation for stable perturbations of multi-germs is slightly more complicated, as can be seen with the example of Reidemeister moves II and III in Section \ref{Reidemeister} below.  Here $D^k(f_t)$ may have more than one connected component, and hence have vanishing alternating homology in dimension $0$ as well as in middle dimension.  As the calculations with Reidemeister moves II and III show, the differentials  
$\pi^k_*:H^\alt_0(D^{p+1}(f_t))\to H_0^\alt(D^{p}(f_t))$ may not all be zero. 

From \eqref{icsscor} it follows that for a stable perturbation of a mono-germ
\beq\label{mooi}\mu_I(f)=\sum_{k=2}^{n+1}\text{rank}\,H_{n-k+1}^\alt(D^k(f_t)).\eeq
In \cite[Theorem 4.6]{Hou97}, Kevin Houston showed that if $f_t$ is a stable perturbation of an $\mathcal A$-finite mono-germ $f_0$, then the {\it alternating} homology of $D^k(f_t)$ is once again confined to middle dimension, even though the ordinary homology of $D^k(f_0)$ may not be \footnote{In fact for the stable perturbation $f_t$ of the germ $(\CC^5,0)\to (\CC^6,0)$  described below, both $D^2(f_t)$ and $D^3(f_t)$ have non-trivial homology below middle dimension.}. From Houston's theorem its follows that \eqref{icsscor} and 
\eqref{mooi} hold for stable perturbations of mono-germs of any corank. 

In both of our examples of corank 2 mono-germs, the multiplicity of $f_0$,
$$\dim_{\CC}\frac{\OO_{\CC^n,0}}{f_0^*\mm_{\CC^{n+1},0}\OO_{\CC^n,0}},$$
is equal to 3, so $f_t$ has no quadruple or higher multiple points, and \eqref{mooi} reduces to 
\beq\label{hgm}\mu_I(f_0)=\text{rank}\,H_{n-1}^\alt(D^2(f_t))+\text{rank}\,H_{n-2}^\alt(D^3(f_t)).\eeq

If $f_0:(\CC^n,0)\to (\CC^{n+1},0)$ is a germ with $\mu_I(f_0)=1$, then \eqref{mooi} 
implies that the vanishing homology of the image comes from just one of the multiple point spaces. It is an interesting project to determine, for each such $f_0$, which one this is. It is possible to show that {\it the answer depends only on the isomorphism class of the local algebra of $f_0$}. It is far from clear to me how to determine the answer from the local algebra.

\subsection{Symmetric group actions on the homology of the multiple point spaces}
From here on, and in the rest of the paper, we will consider only homology with rational coefficients, and by  $H_i(D^k(f_t))$ we will
mean always $H_i(D^k(f_t);\QQ)$.

Each multiple-point space $D^k(f_t)$ is acted upon by the symmetric group $S_k$, permuting the factors of $U_t^k$. The resulting representation of $S_k$ on $H_*(D^k(f_t);\QQ)$ splits as a direct sum of isotypal components, whose ranks are the principle numerical invariants of the disentanglement. We have
$$H_i(D^2(f_t))\simeq H_i^T(D^2(f_t))\oplus H_i^\alt(D^2(f_t)),$$ where the two summands  are the subspaces of $H_i(D^2(f_t))$ on which $S_2$ acts trivially, and by its sign representation, respectively, and 
$$H_i(D^3(f_t))=H_i^T(D_3(f_t))\oplus H_1^\alt(D^3(f_t))\oplus H_1^\rho(D^3(f_t)),$$ where now the summands correspond to the trivial, sign and irreducible degree 2 representation of $S_3$.

Let $M_k(f_0)$ and $M_k(f_t)$ denote the set of {\it target} $k$-tuple points of $f$ and $f_0$ respectively -- points with at least $k$ preimages, counting multiplicity. By e.g. \cite{MP89}, the germ $(M_k(f_0),0)$ is defined by the $(k-1)$'st Fitting ideal of the $\OO_{\CC^{n+1},0}$-module $f_{0*}(\OO_{\CC^n})_0$, that is,  the ideal generated by the $(m-k+1)\times(m-k+1)$ minors of the $m\times m$ matrix of a presentation of $f_{0*}(\OO_{\CC^n})_0$.  

\begin{lemma}\label{conl} Let $f_0:(\CC^n,0)\to (\CC^{n+1},0)$ have multiplicity $k$ and isolated instability, and suppose that $M_k(f_0)$ is non-singular. Let $f_t$ be a stable perturbation of $f_0$. Then $H_i^T(D^k(f_t))=0$ for all $i>0.$
\end{lemma}
\begin{proof} Because the multiplicity of $f_0$ is $k$, $f_t$ has no $(k+1)$-tuple points, and it follows that $M_k(f_t)
\simeq D^k(f_t)/S_k$, and therefore $H_i(M_k(f_t))\simeq H_i^T(D^k(f_t)).$ Because $M_k(f_0)$ is smooth, $M_k(f_t)$ is contractible, and  the result follows.
\end{proof}

Lemma \ref{conl}, with $k=3$,  applies to both of the germs we consider. Smoothness of $M_3(f_t)$ can be seen in each case  by considering a presentation of $f_{0*}(\OO_{\CC^n})_0$.

Suppose $f$ has corank $>1$. We have no closed formula for  generators of the ideal defining $D^k(f)$ for $k\geq 3$, but for $D^2(f)$ there is a formula, introduced in \cite{Mon87}, for germs of any corank. The ideal $(f\times f)^*(I_{\Delta_{n+1}})$ obtained by pulling back the ideal defining the diagonal in $\CC^{n+1}\times\CC^{n+1}$ vanishes on
$D^2(f)$, but also on the diagonal in $\Delta_n\subset\CC^n\times \CC^n$.  To remove $\Delta_n$ and leave only 
the points in the closure of the set of strict double points, we proceed as follows.  The ideal $(f\times f)^*(I_{\Delta_{n+1}})$,
generated by $f_i(x^{(1)})-f_i(x^{(2)})$, for $i=1,\ld, n+1$, is contained in $I_{\Delta_n}$, which is generated by 
$x_j^{(1)}-x_j^{(2)}, j=1,\ld, n$. Thus for $i=1,\ld,n+1$ there are functions $\a_{ij}(x^{(1)},x^{(2)})$ such that
$$f_i(x^{(1)})-f_i(x^{(2)})=\sum_{j=1}^n\a_{ij}(x^{(1)},x^{(2)})\left(x^{(1)}_j-x^{(2)}_j\right).$$
The $(n+1)\times n$ matrix $\a=(\a_{ij})$ restricts to the jacobian matrix of $f$ on $\Delta_n$. We take
$$I_2(f)=(f\times f)^*(I_{\Delta_{n+1}})+\text{min}_n(\a).$$
 
\begin{lemma}\label{gend2} Let $f_0:(\CC^n,0)\to (\CC^{n+1},0)$ be $\mathcal A$- finite and not an immersion. Then 
$D^2(f_0)$, as defined by $I_2(f_0)$, is Cohen-Macaulay of dimension $n-1$, and normal. \eop
\end{lemma}
The proof of Cohen-Macaulayness has been part of the folklore for some time, but has recently been written up carefully by Nu\~no-Ballesteros and Pe\~nafort in \cite{NuPe15}.
When $n=3$, $D^2(f_0)$ is therefore a normal surface singularity, and so by the Greuel-Steenbrink theorem, \cite[Theorem 1]{GrSt81}, $H_1(D^2(f_t))=0$.
\subsection{Calculating $\mu_I(f)$}\label{mui}
Let $f_0:(\CC^n,0)\to (\CC^{n+1},0)$ have finite codimension and let 
$$F:\left(\CC^n\times\CC^d,(0,0)\right)\to \left(\CC^{n+1}\times \CC^d,(0,0)\right), \quad F(x,u)=(f_u(x),u)$$
 be a versal deformation. If $G$ is a reduced equation for the image of $F$ then for $u\in \CC^d$,
 $g_u:=G(\_,u)$ is a reduced equation for the image of $f_u$. By a theorem of Siersma (\cite{siersma}), the image of $g_u$ has the homotopy type of a wedge of $n$-spheres, whose number is equal to the number of critical points of $g_u$ (counting multiplicity) which move off the zero level as $u$ leaves $0$. Note that the number of $n$-spheres is, by definition, the image Milnor number $\mu_I(f_0)$. We can therefore calculate $\mu_I(f_0)$ as follows: define the relative jacobian ideal $J^{\text{rel}}_G$ by
 $$J^{\text{rel}}_G=\left(\frac{\p G}{\p y_1},\ld, \frac{\p G}{\p y_{n+1}}\right)$$
 where $y_1,\ld,y_{n+1}$ are coordinates on $(\CC^{n+1},0)$. The relevant critical points of the functions $g_t$ together make up the residual components of $V(J^{\text{rel}}_G)$ after removal of its components lying in $\{G=0\}$. This residual set can be found as the zero-locus of the saturation  $(J^{\text{rel}}_G:G^\infty)$, defined as 
 $$\bigcup_{k\in\NN}\{h\in \OO_{\CC^{n+1}\times\CC^d,(0,0)}: hG^k\in J^{\text{rel}}_G\}.$$
 
 We denote the zero locus of $(J^{\text{rel}}_G:G^\infty)$ by $\Sigma$. Thus the image Milnor number $\mu_I(f_0)$ is the degree of the projection $(\Sigma,0)\to( \CC^d,0)$. This degree can be calculated as the intersection number $\left(\Sigma, \CC^{n+1}\times\{0\}\right)_{(0,0)}$.  If $\Sigma$ is 
 Cohen-Macaulay then
 \beq\label{keq}
 \mu_I(f_0)=\left(\Sigma, \CC^{n+1}\times\{0\}\right)_{(0,0)}=\dim_{\CC}\left(\frac{\OO_{\CC^{n+1}\times \CC^d,(0,0)}}
 {(J^{\text{rel}}_G:G^\infty)+(u_1,\ld,u_d)}\right)
 \eeq
 where $u_1,\ld,u_d$ are coordinates on $(\CC^d,0)$. 
 
 In both of the examples considered here, this is the case, and it is a straightforward {\it Macaulay2} calculation to follow this procedure (including to check the Cohen-Macaulayness of $\Sigma$) and find $\mu_I(f_0)$: it is $18$ for the germ $(\CC^3,0)\to (\CC^4,0)$, and $1$ for the germ $(\CC^5,0)\to (\CC^6,0)$. 
 
 If $\Sigma$ is not Cohen-Macaulay, the intersection number can be calculated using Serre's {\it formule clef}, \cite{serre}, which we use to calculate a related intersection number in Subsection \ref{a1d2} below. 
 
\begin{remark}{\em The method outlined here gives no hint to any relation between $\mu_I(f_0)$ and the ${\mathcal A}_e$-codimension
 of $f_0$. It is conjectured that provided $(n,n+1)$ are nice dimensions, the standard ``Milnor-Tjurina" relation holds, namely
 \beq\label{miltjur}{\mathcal A}_e\text{-codim }f_0\leq \mu_I(f_0)\eeq
 with equality if $f_0$ is weighted homogeneous. In \cite{mond15} another slightly more complicated method for calculating $\mu_I$ is explained, with a similar case-by-case justification --  verification of the Cohen Macaulayness of  a certain relative $T^1$ module,  $T^{1\ \text{rel}}_{{\mathcal K}_{h,e}}i$,  and consequent conservation of multiplicity. The virtue of this second method is that the relation \eqref{miltjur} is an immediate consequence, whenever Cohen-Macaulayness of the relative $T^1$ can be shown, since $T^1_{{\mathcal A}_e}f_0$ is a quotient of  $T^{1}_{{\mathcal K}_{h,e}}i_0$
 }
 \end{remark}

\subsection{Examples}\label{Examples}
\subsubsection*{Example I: the ICSS for a stable perturbation of $f(x,y)=(x,y^3,xy+y^5)$.}
Here we apply the calculations described in \ref{althom}
to a stable perturbation of the germ of the title of this subsection, of type $H_2$. For any map-germ $f:(\CC^2,0)\to (\CC^3,0)$ of the form 
$f(x,y)=(x,f_2(x,y), f_3(x,y))$,  $D^2(f_t)$ is defined in $(x,y_1,y_2)$-space by the equations (see \cite{marar-mond})
\beq\label{d2h2}
\frac{\left|\begin{array}{cc}1&f_i(x,y_1)\\1&f_i(x,y_2)\end{array}\right|}{\left|\begin{array}{cc}1&y_1\\1&y_2\end{array}\right|}\quad i=2, 3
\eeq 
and $D^3(f)$ is defined in $(x,y_1,y_2,y_3)$-space by the equations
\beq\label{d3h2}
\frac{\left|\begin{array}{ccc}1&f_i(x,y_1)&y_1^2\\1&f_i(x,y_2)&y_2^2\\1&f_i(x,y_3)&y_3^2\end{array}\right|}{\left|\begin{array}{ccc}1&y_1&y_1^2\\1&y_2&y_2^2\\1&y_3&y_3^2\end{array}\right|},\quad
\frac{\left|\begin{array}{ccc}1&y_1&f_i(x,y_1)\\1&y_2&f_i(x,y_2)\\1&y_3&f_i(x,y_3)\end{array}\right|}{\left|\begin{array}{ccc}1&y_1&y_1^2\\1&y_2&y_2^2\\1&y_3&y_3^2\end{array}\right|}\quad i=2, 3.
\eeq
In this case these give
$$y_1^2+y_1y_2+y_2^2,\quad x+y_1^4+y_1^3y_2+y_1^2y_2^2+y_1y_2^3+y_2^4$$
for $D^2(f)$ and
$$P_2(y_1,y_2,y_3),\quad y_1+y_2+y_3,\quad x+ P_4(y_1,y_2,y_3),\quad P_3(y_1,y_2,y_3)$$
for $D^3(f)$, where each $P_j$ is a symmetric polynomial of degree $j$. Thus $D^2(f)$ is an $A_1$ curve singularity and
$D^3(f)$ is a non-reduced point of multiplicity 6. If $f_t$ is a stable perturbation then $D^2(f_t)$ is a Milnor fibre of the $A_1$ singularity, homotopy equivalent to a circle, and $D^3(f_t)$ consists of 6 points forming a single $S_3$-orbit. By judicious choice of parameter values $u$ and $v$ in the miniversal deformation $f_{u,v}(x,y)=(x,y^3+uy,xy+y^5+vy^2)$ (see \cite{grp}), one can arrange that the {\it real} picture of $D^2(f_t)$ and $D^3(f_t)$, and their projections $D^3_1(f_t), D^2_1(f_t)$, are as shown in the following diagram. 
\vs
\cl{\ig[width=4in]{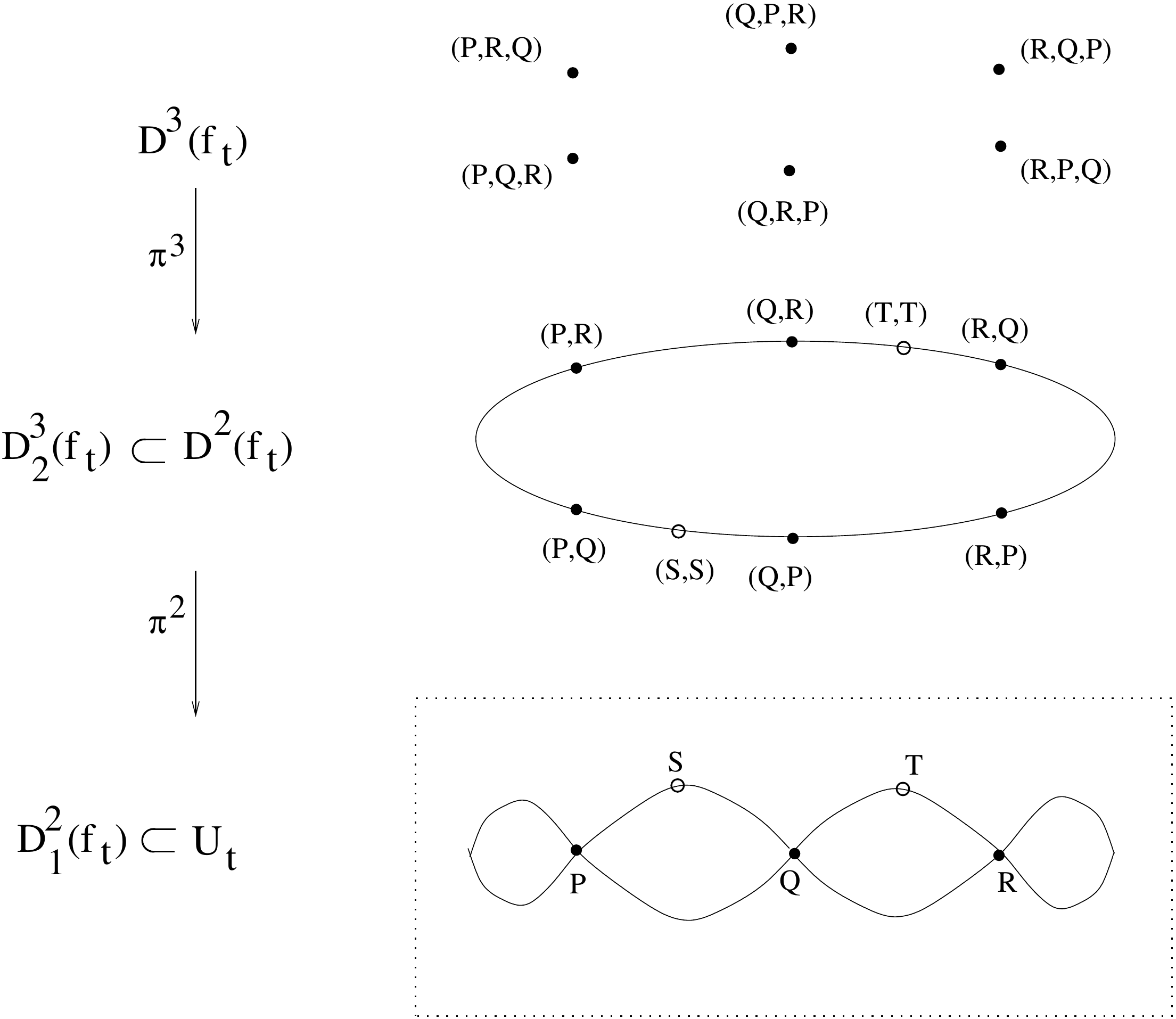}}
\vs
Here $S$ and $T$ in $U_t$ are the non-immersive points of $f_t$. At each, the germ of $f_t$ is equivalent to the parametrisation of the Whitney umbrella, $(x,y)\mapsto (x,y^2,xy)$, since this is the only stable non-immersive germ in this dimension range. The non-strict double points $(S,S)$ and $(T,T)$ are the fixed points of the involution $(1,2)$ on $D^2(f_t)$, which, in our picture, is induced by a reflection in the straight line joining them. 

As a single faithful $S_3$-orbit, $D^3(f_t)$ carries an alternating cycle, 
$$c_0=(P,Q,R)-(P,R,Q)+(R,P,Q)-(R,Q,P)+(Q,R,P)-(Q,P,R).$$
The projection of this cycle to $D^2(f_t)$, $\pi^3_\#(c_0)$, is an alternating boundary: for instance
$$\pi^3_\#(c_0)= (P,Q)-(P,R)+(R,P)-(R,Q)+(Q,R)-(Q,P)=\p(c_1)$$
where $c_1$ is the alternating 1-chain 
$$[(T,T)(Q,R)]-[(T,T)(R,Q)]+[(Q,P)(R,P)]-[(P,Q)(P,R)]$$
(here, for any two non-antipodal points $A,B\in D^2(f_t)$, $[A,B]$ denotes the singular 1-simplex parametrising the shorter arc from
$A$ to $B$). 
The projection of $c_1$ to $U_t$ is a 1-cycle in
$U_t$, and, as can be seen, is the boundary of the 2-chain $c_2$ shown in the diagram as the shaded region. And by
the argument above,
$f_{t\#}(c_2)$ is a cycle in the image $X_t$, indeed one of the two generators of $H_2(X_t)$.  Another generator comes from the alternating 1-cycle $c_1'$ on $D^2(f_t)$ consisting of the anticlockwise arc $[(S,S)(T,T)]$ minus the clockwise arc $[(S,S)(T,T)]$. 
I encourage the reader to find a 2-chain $c_2'$ on on $U_t$ such that $\p c_2'=\pi^2_\#(c_1').$ 

\subsubsection*{Example II: the Reidemeister moves}\label{Reidemeister}
The Reidemeister moves of knot theory are versal deformations of the three ${\mathcal A}_e$-codimension 1 singularities of mappings from the line to the plane. It is instructive to look at their disentanglements (in the sense described above), and at the resulting ICSS. 
The codimension 1 germs are shown in the middle column of the table below, and the right hand column shows a 1-parameter versal deformation, which, fixing $t\neq 0$, gives a stable perturbation. 
\beq\label{rei123}
\begin{array}{|l||c|c|}
\hline
&&\\
I&f_0:x\mapsto (x^2,x^3)&f_t: x\mapsto (x^2,x^3-tx)\\
\hline
&&\\
II&f_0:\left\{\begin{array}{l}x\mapsto (x,\ \ x^2)\\x'\mapsto (y,-y^2)\end{array}\right.&f_t:\left\{
\begin{array}{l}x\mapsto (x, x^2-t)\\y\mapsto (y,-y^2)\end{array}\right.\\
\hline
&&\\
III&f_0:\left\{\begin{array}{l}x\mapsto (x,\ \ x)\\
y\mapsto (y,\ \ 0)\\
z\mapsto (z,-z)
\end{array}\right.&f_t:\left\{\begin{array}{l}x\mapsto (x,\ \ x)\\
y\mapsto (y,\ \ t)\\
z\mapsto (z,-z)
\end{array}\right.\\
\hline
\end{array}
\eeq
For all three cases, the non-trivial modules in the $E^1$ page of the ICSS for $f_t$ are contained in the single column
\beq\label{icrm}\xymatrix{0\ar[d]\\H_0^\alt(D^3(f_t))\ar[d]^{\pi^3_*}\\H_0^\alt(D^2(f_t))\ar[d]^{\pi^2_*}\\H_0(U_t)\ar[d]\\0}\eeq
\subsection*{Reidemeister I}
Take $F:(x,t)\mapsto (t,f_t(x))$ as stable unfolding. Since in order that $F(t_1,x_1)=F(t_2,x_2)$, we must have $t_1=t_2$, we can embed $D^2(F)$ in $\CC^3$ with coordinates $t,x_1,x_2$. There, following the recipe preceding Lemma \ref{gend2} above, we find that
$D^2(F)$ is defined by the equations 
\beq\label{D2I}\frac{x_1^2-x^2_2}{x_1-x_2}=\frac{ x_1^3-tx_1-(x_2^3-tx_2)}{x_1-x_2}=0.\eeq
 Simplifying, this gives 
\beq\label{D2Ia}x_1+x_2=0\quad x_1^2+x_1x_2+x_2^2=t.\eeq
Thus $D^2(f_0)$ is a 0-dimensional $A_1$ singularity. Setting $t>0$ for a good real picture, and denoting $\sqrt{t}$ by $P$
and $-\sqrt{t}$ by $Q$, $D^2(f_t)$ is its Milnor fibre, the point-pair $\{(P,Q), (Q,P) \}$. Then for $t\neq 0$,
$H_0^\alt(D^2(f_t))\simeq\QQ$, generated by the class of $[(P,Q)]-[(Q,P)]$. Note that $H_0^\alt(D^2(f_0))=0$, since when
$t=0$, $P=Q$. 
For both $f_0$ and $f_t$, $D^3=\emptyset$.

In the $E^1$ page \eqref{icrm}, $H_0^\alt(D^3(f_t))=0$. We have $\pi_*^2=0$, for $\pi^2_*\bigl([(P,Q)]-[(Q,P)]\bigr)=[P]-[Q]$, and $U_t$ is connected, so that $[P]=[Q]$.  Hence the spectral sequence collapses at $E^1$, and for $t\neq 0$
$$H_0(X_t)=H_0(U_t)=\QQ, \quad H_1(X_t)=H_0^\alt(D^2(f_t))=\QQ.$$ 
\subsection*{Reidemeister II} Here both branches of the bi-germ $f_0$ are immersions, so all multiple points
are strict.  Denote by $0_x$ and $0_y$ the origins of the coordinate systems with coordinates $x$ and $y$ respectively.  The domain of the stable perturbation $f_t$ is a disjoint union $U_t=U_{x,t}\cup U_{y,t}$, where $U_{x,t}$ is a contractible
neighbourhood of $0_x$ and $U_{y,t}$ is a contractible neighbourhood of $0_y$. Thus $H_0(U_t)\simeq\QQ^2$. There are no triple points, and $D^2(f_t)$ consists of
\beq\label{r2d2}
\{(x,y)\in (\CC,0_x)\times(\CC,0_y):x=y,  x^2-t=-y^2\}\eeq
together with its image under the involution $(1,2)$ sending $(x,y)$ to $(y,x)$. When $t=0$ this is a pair of
$0$-dimensional $A_1$ singularities, interchanged by $(1,2)$. To describe $D^2(f_t)$ for $t\neq 0$, denote the points in $(\CC,0_x)$ with $x$ coordinates $\sqrt{t/2}$ and $-\sqrt{t/2}$ by $P_x$ and $Q_x$ respectively, and the points in $(\CC,0_y)$ with $y$ coordinates $\sqrt{t/2}$ and $-\sqrt{t/2}$ by $P_y$  and $Q_y$.
Then for $t\neq 0$,
\beq D^2(f_t)=\{(P_x, P_y), (P_y, P_x), (Q_x, Q_y), (Q_y, Q_x)\},\eeq
with the involution $(1,2)$ interchanging the first and second points, and the third and fourth.
For $t=0$, this collapses just to 
$$D^2(f_0)=\{(0_x,0_y),(0_y,0_x)\}.$$
Thus for $t\neq 0$, $H_0^\alt(D^2(f_t))$ is two-dimensional, with basis
$[(P_x,P_y)]-[(P_y,P_x)],\ [(Q_x,Q_y)]-[(Q_y,Q_x)]$, and for $t=0$, $H_0^\alt(D^2(f_0))$ has basis 
$[(0_x,0_y)]-[(0_y,0_x)]$. 
With respect to the basis of $H^\alt_0(D^2(f_t))$ described above, and the basis $[P_x], [P_y]$ for $H_0(U_t)$,  
$\pi^2_*$
 has matrix $\pmat 1&1\\-1&-1\emat$ when $t\neq 0$, and thus has 1-dimensional kernel and cokernel. The spectral sequence collapses at $E^2$, and 
 $$H_1(X_t)=E^2_{1,0}=\ker \pi^2_*\simeq\QQ,\quad  H_0(X_t)=E^2_{0,0}=\coker \pi^2_*\simeq\QQ.$$
\subsection*{Reidemeister III}
We use the same conventions as for Reidemeister II.  Let $P_x$ and $P_y$ denote the points in $U_{x,t}$ and $U_{y,t}$ with $x$ and $y$ coordinate  $t$, and let $Q_y$ and $Q_z$ denote the points in $U_{y,t}$ and $U_{z,t}$ with $y$ and $z$ coordinate $-t$. Note that when $t=0$, then $P_x=0_x$, etc. Then
\beq
\begin{split}
D^2(f_t)\cap \left(U_{x,t}\times U_{y,t}\right)=\{(P_x,P_y)\}\\
D^2(f_t)\cap \left(U_{x,t}\times U_{z,t}\right)=\{(0_x,0_y)\}\\
D^2(f_t)\cap \left(U_{y,t}\times U_{z,t}\right)=\{(Q_x,Q_z)\}
\end{split}
\eeq
and 
$$D^3(f_0)\bigcap U_{x,t}\times U_{y,t}\times U_{z,t}=\{(0_x,0_y,0_z)\}.$$
Thus
\beq
\begin{array}{cc}H_0^\alt(D^3(f_0))\simeq\QQ &H_0^\alt(D^3(f_t))=0\\
\\
H_0^\alt(D^2(f_0))\simeq\QQ^3&H_0^\alt(D^2(f_t))\simeq\QQ^3\\
\\
H_0(U_0)\simeq\QQ^3&H_0(U_t)\simeq\QQ^3
\end{array}
\eeq
with bases shown in the following table.
$$
\begin{array}{|c|c|}
\hline
\text{Module}&\text{Basis}\\
\hline
&\\
H_0^\alt(D^3(f_0))&[(0_x,0_y,0_z)]-[(0_x,0_z,0_y)]+[(0_z,0_x,0_y)]-[(0_z,0_y,0_x)]+[(0_y,0_z,0_x)]-[(0_y,0_x,0_z)]\\
\hline
&\\
H_0^\alt(D^2(f_t))&[(P_x,P_y)]-[(P_y,P_x)],\ -[(0_x,0_z)]+[(0_z,0_x)],\ [(Q_y,Q_z)]-[(Q_z,Q_y)]\\
\hline
&\\
H_0(U_t)&[0_x]=[P_x],\ [P_y]=[Q_y],\ [Q_z]=[0_z]\\
\hline
\end{array}
$$
With respect to these bases, the differentials $\pi^k_*$ have the following matrices (with the first only for $t=0$):
$$
\begin{array}{|c||c|}
\hline 
&\\
\pi^3_*=\pmat 1\\1\\1\emat&
\pi^2_*=\pmat 1&-1&0\\-1&0&1\\0&1&-1\emat\\
\hline
\end{array}
$$
In the spectral sequence for $f_0$, the image of $\pi^3_*$ kills the kernel of $\pi^2_*$. When $t\neq 0$,
$D^3$ vanishes, along with its homology, while $H_0^\alt(D^2(f_t))$ remains unchanged. The spectral sequence collapses at $E^2$, and
$$H_1(X_t)=E^2_{1,0}=\ker \pi^2_*\simeq\QQ,\quad  H_0(X_t)=E^2_{0,0}=\coker \pi^2_*\simeq\QQ.$$  
\section{New examples: disentanglements of two germs of corank 2}
\subsection{Summary of results}
\subsubsection{$\mathbf {f_0:(\CC^3,0)\to (\CC^4,0),\quad f_0(x,y,z)=(x,y^2+xz+x^2y,yz,z^2+y^3)}$}

The rows of the following table show relations between the ranks of the isotypal subspaces of the homology groups of $D^2(f_t)$ and $D^3(f_t)$ and of the homology groups of their projections to $U_t$, $D^2_1(f_t)$ and $D^3_1(f_t)$, and $VD_\infty$, the number of Whitney umbrellas on $D^2_1(f_t)$, which plays a crucial role in our calculation. The left hand column shows where in the paper the calculation is made. Blank spaces indicate zeros.
\def\aalt{\text{\tiny Alt}}

\beq\label{bigmat}
\begin{array}{|r || c | c | c | c | c | c | c | c || l|}
\hline
\text{datum}&\hskip -.1cm H_2^T(D^2)\hskip -.1cm&\hskip -.1cm H_2^\aalt(D^2)\hskip -.1cm&\hskip -.1cm H_2(D^2_1)\hskip -.1cm&\hskip -.1cm H_1^T(D^3)\hskip -.1cm&\hskip -.1cm H_1^\aalt(D^3)\hskip -.1cm&\hskip -.1cm H_1^\rho(D^3)\hskip -.1cm&\hskip -.1cm 
H_1(D^3_1)\hskip -.1cm&\hskip -.1cm VD_\infty\hskip -.1cm&\\
\hline
\hline
\eqref{hgm}& &1 & & &1 & & &  &=18  \\
\hline
\delta(D^3_1)\text{ in \S 4.2}& & & & & & &1 & &=8 \\
\hline
\delta(M_3)\text{ in \S 4.2}& & & &1 & & & & & =0  \\
\hline
\eqref{chid}& &  & &1 &1 &1 &-2 &  -1&=-1  \\
\hline
\eqref{vd}& &  & & & & & &1 &= 10 \\
\hline
\eqref{miss}& & &  & -1&1 & &  &-1 &=-1  \\

\hline
\eqref{sier}& & & 1 & & & & &  &=27 \\
\hline
\eqref{icss}-\eqref{3eq}&1 & 1&-1& &1 &\frac{1}{2} & &&=0\\
\hline
\end{array}
\eeq
The rank of the matrix of coefficients is 8, so we are able to compute all of the invariants. The following table shows their values.
\beq
\label{conc}
\begin{array}{| c | c | c | c | c | c | c |  c |}
\hline
\hskip -.1cm H_2^T(D^2)\hskip -.1cm&\hskip -.1cm H_2^\aalt(D^2)\hskip -.1cm&\hskip -.1cm H_2(D^2_1)\hskip -.1cm&\hskip -.1cm H_1^T(D^3)\hskip -.1cm&\hskip -.1cm H_1^\aalt(D^3)\hskip -.1cm&\hskip -.1cm H_1^\rho(D^3)\hskip -.1cm&\hskip -.1cm 
H_1(D^3_1)\hskip -.1cm&\hskip -.1cm VD_\infty\hskip -.1cm\\
\hline
\hline
1&9&27&0&9&16&8&10\\
\hline

\end{array}
\eeq
 \subsubsection{$\mathbf{f_0:(\CC^5,0)\to (\CC^6,0),\quad f_0(x,y,a,b,c)=(x^2+ax+by,xy,y^2+cx+ay,a,b,c)}$}
We are able to show
 \ben
 \y $H_3^\alt(D^3(f_t))\simeq\QQ$ and $H^\alt_4(D^2(f_t))=0$, so the vanishing homology of the image comes from the triple points.
 \y
$H_1(D^3(f_t))=0$,
$H_2(D^3(f_t))=H^\rho_2(D^3(f_t))\simeq\QQ^2$, and $H_3(D^3(f_t))=H_3^\alt(D^3(f_t))\simeq \QQ$. 
\y
 $H_1(D^2(f_t))=0$, $H_2(D^2(f_t))=H_2^T(D^2(f_t))\simeq \QQ$. 
  \y
 $\dim_{\QQ}H_4(D^2(f_t))=\dim_{\QQ}H_3(D^2(f_t))\leq 1.$ Both groups are $S_2$-invariant, by Houston's theorem 
 \cite[Theorem 4.6]{Hou97}
 \een
 (a) and (b) are shown in Subsection  \ref{D35to6}, and (c) and (d) are shown in Subsection \ref{defdub}.

This is the first example I know of a stable perturbation of a map-germ $f_0:(\CC^n,0)\to (\CC^{n+1},0)$ for which the vanishing homology of the multiple point spaces is not confined to middle dimension, though of course 
many such examples are to be expected when $f_0$ has corank $>1$. 

\section{Calculations for the germ $(\CC^3,0)\to (\CC^4,0)$}\label{3to4}
\subsection{Triple points}
Since no closed formula for a set of generators for the ideal defining $D^3(f)$ in $(\CC^3)^3$ is known,  we do not have direct access to any of the invariants of $D^3(f_t)$. However, we are able to build up a complete picture of the representation of $S_3$ on its
homology, and in particular calculate the dimension of $H_1^\alt(D^3(f_t))$, by working our way up from its image
under projection to $U_t$, $D^3_1(f_t)$. 
\begin{lemma}\label{list} $D^3_1(f_t)$ is a smoothing 
of $D^3_1(f)$
\end{lemma}
\begin{proof}
We have to show both that $D^3_1(f_t)$ is smooth, and that it is the fibre of a flat deformation of $D^3_1(f)$. The first statement is a consequence of the classification of stable map-germs. Up to $\mathcal A$-equivalence, the only stable germs of maps
$\CC^3\to\CC^4$ are 
\ben
\y
a trivial unfolding of the parameterisation of the Whitney umbrella:
$p_1(u,v,w)=(u,v,w^2,vw);$
\y
a bi-germ whose two branches are a germ of type (a) and an immersion, meeting in general position in $\CC^4$;
\y
a multi-germ consisting of $k$ immersions meeting in general position, for $k=1,2,3,4$ (we denote these by (c1),
\ldots,(c4)).
\een
Since $f_t$ is stable, every one if its germs is one of these types, and one can easily check that for each of them, except for (c4), the triple point locus $D^3_1$, where non-empty, is smooth. In the mapping $f_t$ there are no points of type (c4), so $D^3_1f_t)$ is smooth.

For the second statement, let $F:(\CC^3\times S,(0,0)\to (\CC^4\times S,(0,0)$ be a stable unfolding of $f$ over a smooth base $S$. Then $D^3(F)$ has dimension $1+\dim\,S$.  By the principle of iteration,  $D^3_1(F)=M_2(\pi^2:D^2(F)\to \CC^3\times S)$ (where $M_2$ means the set of double points in the target). Now $D^2(F)$ is Cohen-Macaulay, and $\pi^2$ is finite and generically 1-1, so $M_2(\pi^2)$ is also Cohen Macaulay (\cite{MP89}). Flatness of the projection $D^3_1(F)\to S$ now follows from the fact that the dimension of its fibre, $D^3_1(f_t)$, is equal to $\dim\, D^3_1(F)-\dim\, S$.  
\end{proof}
It follows from the lemma that $\text{rank}\ H_1(D^3_1(f_t))=\mu(D^3_1(f),0)$. We find $\mu$ by using Milnor's formula $\mu=2\delta-r+1$ (\cite{Milnor}), where 
$\delta$ is the $\delta$-invariant of a curve-germ and $r$ the number of its branches. 
We find $D^3_1(f)$ as the zero locus of the ideal
$f^*(\fitt_2)$, where $\fitt_2:=\fitt_2(f_*\OO_{\CC^3,0}))$ is the second Fitting ideal\footnote{For an $R$-module 
$M$, $\fitt_k(M)$ is the ideal generated by the minors of co-size $k$ of the matrix of a presentation of $M$.} of $\OO_{\CC^3,0}$ considered as $\OO_{\CC^4,0}$-module via $f^*$. 
{\it Macaulay2} gives the following presentation of $f_*(\OO_{\CC^3})$:
\beq\label{pres4}\begin{pmatrix}
-X^2U^2-2XUV+V^2-UW&X^4+U^2+X^3V&X^3U+2X^2V+XW\\
X^4+U^2+X^3V&-X^6-2X^2U-XV-W&-X^5-XU+V\\
X^3U+2X^2V+XW&-X^5-XU+V&-X^4-U
\end{pmatrix}
\eeq
from which we see that 
\beq\label{F2}\fitt_2=(X^4+U,V,X^2U+W)\eeq and   
\beq\label{ff2}f^*\fitt_2=(x^4+x^2y+y^2+xz,yz,x^3z+z^2).\eeq

Primary decomposition of the ideal \eqref{ff2} shows that the curve $D^3_1(f)$ has three smooth components:
$$C_1=V(y,x^3+z),\quad 
C_2=V(z, y-\xi x^2)
\quad 
C_3=V(z,y-\xi^2x^2)$$
where $\xi=e^{2i\pi/3}$, with parameterisations
$$\g_1(t)=(t,0,-t^3),\quad \g_2(u)=(u,\xi u^2,0),\quad \g_3(v)=(v,\xi^2v^2,0).$$ 
Denoting by $\OO_\sig:=\CC\{t\}\oplus
\CC\{u\}\oplus\CC\{v\}$ the ring of the normalisation of $\Sigma$, we find that
$$n^*(\OO_{\Sigma,0})=(t^3)\oplus (u^3)\oplus (v^3)
+\text{Sp}\{(1,1,1),(t,u,v),(t^2,u^2,v^2), (0,\xi u^2, \xi^2v^2)\}\subset
\OO_\sig.$$
Hence 
$$\delta(D^3_1(f))=\dim_{\CC}\frac{\OO_\sig,0}{n^*\OO_{\Sigma,0}}=5,$$
and $$\text{rank }H_1(D^3_1(f_t))=\mu(D^3_1(f))=2\delta-3+1=8.$$

The projection $D^3(f_t)\to D^3_1(f_t)$ is a double cover, with points
$(a,b,c)$ and $(a,c,b)$ sharing the same image, and no points of higher multiplicity, as $f$ has no quadruple points.
The cover is simply branched at triple points of the form $(a,b,b)$; there are no triple points of the form $(a,a,a)$, since if there were, then $f_t$ would have multiplicity $\geq 3$ at $a$, and as we see in the list of stable germs in the proof of Lemma \ref{list}, none has multiplicity $>2$.  Thus 
\beq\label{chid}\chi(D^3(f_t))=2\chi(D^3_1(f_t))-\#\text{branch points}=-14-\#\text{branch points}.\eeq
To complete the calculation of the Euler characteristic of $D^3(f_t)$, we have to compute the number of branch points. This seems not to be straightforward. Though the branch points are points of intersection of 
$D^3_1(f_t)$ and the non-immersive locus $R(f_t)$, both of these are curves so their intersection in $U_t$ is
not a proper intersection. Both curves lie in the surface $D^2_1(f_t)$, where the intersection is proper, but $D^3_1(f_t)$ is the singular locus of $D^2_1(f_t)$ and so again calculation of the intersection number is difficult.
Instead we use the fact that the branch points are Whitney umbrella points of $D^2_1(f_t)$, which we explain in the next section, and count them using a theorem of Theo de Jong in \cite{dJ90}.

\subsection{Double points}\label{a1d2}
\begin{lemma}\label{beg} $(a,b,b)\in D^3(f_t)$ if and only if $(a,b)\in D^2(f_t)$ is 
a Whitney umbrella point of the projection $\pi^2:D^2(f_t)\to U_t$.
\end{lemma}
\begin{proof} The map $$(a,b,c)\mapsto ((a,b),(a,c))$$ identifies $D^3(f_t)$ with $D^2(\pi^2:D^2(f_t)\to U_t)$. A point of the form $(a,b,b)$ becomes a fixed point of the involution $((a,b),(a,c))\mapsto ((a,c), (a,b))$, and thus
a non-immersive point of $\pi^2$. By Remark \ref{inter}, this must be a Whitney umbrella point. \end{proof}
From Lemma \ref{beg} we see that to find the dimension of $H_1^\alt(D^3(f_t))$ we must count the number of
Whitney umbrellas on $D^2_1(f_t)$.  Let $W(D^2_1(f_t))$ denote the set of all such points. They appear when $f_t$ has a bi-germ of type (b) in the list in the proof of Lemma \ref{list}: the Whitney umbrella appears on $D^2_1(f_t)$ at the  source point of the immersive member of the bi-germ.  If $R(f_t)$ is the set of non-immersive points of $f_t$, then $W(D^2_1(f_t))=D^3_1(f_t)\cap R(f_t),$ 
so one might hope to calculate the number of points in $W(D^2_1(f_t))$ as an intersection number. But as remarked above, the intersection is improper: both $D^3_1(f_t)$ and $R(f_t)$ are curves. We are forced to look further afield, and use a theorem of Theo de Jong
(\cite{dJ90}). The {\it virtual number of $D_\infty$ points} on a germ
of singular surface $(S,x_0)\subset\CC^3$, with 1-dimensional singular locus $
\Sigma$, and with reduced equation $h$, is defined as 
follows. Let $\theta(h)$ be the restriction to $\Sigma$ of the germs of 
vector fields on $(\CC^3,x_0)$ tangent to all level sets of $h$. Then 
$\theta(h)\subset\theta_\Sigma$.  
Let $\sig$
be the normalisation of $\Sigma$. Since vector fields lift uniquely
to the normalisation we can consider the quotient $\theta_{\tilde\Sigma,\tilde x_0}/\theta(h)$. De Jong defines
\beq\label{def}
\text{VD}_\infty(S)=\dim_\CC\left(\frac{\theta_{\tilde\Sigma_{\tilde x_0}}}
{\theta(h)}\right)-3\delta(\Sigma)
\eeq
and shows (\cite[Theorem 2.5]{dJ90}) that $VD_\infty(S)$ is conserved in a flat deformation of $S$ which induces a flat deformation of $\Sigma$. 

Let us apply this to the case where $S$ is the surface $D^2_1(f)$ for a finitely
determined map-germ $f:(\CC^3,0)\to (\CC^4,0)$. In this case $\Sigma=D^3_1(f)$.
A deformation of $f$ over a smooth base $S$ induces a flat deformation of the $D^2_1(f)$, since it is  a hypersurface. We have already seen that $D^3_1(f)$ deforms flat over $S$. Thus we may apply de Jong's theorem. The special points on $D^3_1(f_t)$,
where $D^2_1(f_t)$ is not a normal crossing of two sheets, are of two types:
Whitney umbrella points and triple points. We have already seen how
Whitney umbrella points arise; triple points correspond to quadruple 
points of $f_t$, in which four pieces of $\CC^3$ are mapped immersively and
in general position. We denote the number of these by $Q$. 
\begin{corollary} $\text{V}D_\infty(D^2_1(f))= |\text{Fix}(1,2)|-8Q$
\end{corollary}
\begin{proof} Each Whitney umbrella point contributes $1$ to $\vd(D^2_1)$. Each quadruple point gives rise to four triple points on 
$D^2_1(f)$. Each triple point contributes $-2$ to $\vd(D^2_1))$ (\cite[Example 2.3.3]{dJ90}). So
$$\text{V}D_\infty(D^2_1(f))=\#\text{Whitney umbrellas} -2\#\text{triple points}=|\text{Fix}(1,2)|-8Q.$$
\end{proof}
Now we return to the map germ $f$ of Sharland that is the focus of our interest. 

To compute $VD_\infty(D^2_1)$ we need to find lifts to the normalisation $\sig$ of $D^3_1(f)$ of the vector fields annihilating the equation $h$ of $D^2_1(f)$. A {\it Macaulay} calculation finds that modulo the defining ideal of
$D^3_1(f)$, these vector fields are generated by
$$\chi_1=(x^3y^2+2xy^3-3xz^2)\frac{\p}{\p x}+ 
        (2x^2y^3+4y^4)\frac{\p}{\p y}  
     -9z^3 \frac{\p}{\p z}  $$
                
 $$ \chi_2= (y^4+x^2z^2)\frac{\p}{\p x}
         -(2x^3y^3-2xy^4)\frac{\p}{\p y}  +
       3xz^3  \frac{\p}{\p z}     $$
These lift to $$\tilde\chi_1=\left(-3t^7\frac{\p}{\p t}, (2+\xi^2)u^7\frac{\p}{\p u}, (2+\xi)v^7\frac{\p}{\p v}\right), \quad
\tilde\chi_2=\left(t^8\frac{\p}{\p t}, \xi u^8\frac{\p}{\p u}, \xi^2v^8\frac{\p}{\p v}\right) $$ 
in $\theta_\sig=\CC\{t\}\p_t\oplus\CC\{u\}\p_u\oplus
\CC\{v\}\p_v$.  
The $\OO_{\CC^3}$-submodule of $\theta_\sig$ they generate is equal to 
$$(t^{10})\p_t\oplus(u^{10})\p_u\oplus(v^{10})\p_v+\text{Sp}_{\CC}\{\tilde\chi_1,\ x\tilde\chi_1,\ x^2\tilde\chi_1,\ \tilde\chi_2,\ x\tilde\chi_2\}.
$$
Hence $\dim_{\CC}(\theta_\sig/\theta(h))=25$ so that 
\beq\label{vd}
VD_\infty(D^2_1)=25-3\times 5=10.
\eeq

We have proved 
\begin{lemma} The involution $(2,3)$ has $10$ fixed points on $D^3(f_t)$.\eop
\end{lemma}
\begin{corollary}\label{dim3}$\dim_{\CC}H_1(D^3(f_t);\CC)=25.$
\end{corollary}
\begin{proof}
By the lemma and \eqref{chid}, $\chi(D^3(f_t))=-24.$
\end{proof}
\begin{proposition}\label{dimde} $\dim_{\CC}H_1^\alt(D^3(f_t))=9$, $\dim_{\CC}H_1^\rho(D^3(f_t))=16,$ and
$\dim_{\CC}H_2^\alt(D^2(f_t))=9.$
\end{proposition}
\begin{proof}
We use the Lefschetz fixed point theorem:
$$10=\#\text{fixed points of }(2,3) =\sum_{k\geq 0}(-1)^k\left(\text{trace}(2,3)_*:H_k(D^3(f_t))\to H_k(D^3(f_t))\right)$$
\beq\label{miss}=1-\text{trace}(2,3)_*:H_1(D^3(f_t)\to H_1(D^3(f_t))=1+\dim_{\CC}H_1^\alt(D^3(f_t))-\dim_{\CC}H_1^T(D^3(f_t)).\eeq
The last equation here follows from the fact that the trace of $(2,3)$ on the irreducible sign representation of $S_3$, on the trivial representation and on the irreducible 2-dimensional representation is $-1$, $1$ and $0$ respectively.
It is straightforward to check that each fixed point of $(2,3)$ is non-degenerate and therefore has Leftschetz number 1.  Since $H_1^T(D^3(f_t))=0$, we obtain the first equality in the statement of the corollary. The second equality now follows by Corollary \ref{dimde} and the third by the fact that $18=\mu_I(f)=\dim_{\CC}H_2^\alt(D^2(f_t))+\dim_{\CC}H_1^\alt(D^3(f_t))$.\end{proof}

Now we compute $H_2(D^2(f_t))$.
Although we have a formula for the ideal defining $D^2(f)$, we have no method of deriving from it a formula for
the rank of the homology of $D^2(f_t)$. So once again we proceed indirectly, by calculating the homology of the image of its projection to $U_t$, 
$D^2_1(f_t)$. 
\begin{lemma}\label{nocm}
$D^2_1(f_t)$ has the homotopy type of a wedge of 27 2-spheres. 
\end{lemma} 
\begin{proof}
We use the technique explained in Subsection \ref{mui}, based on Siersma's theorem (\cite{siersma}) that the rank of the vanishing homology of $D^2_1(f)$ is equal to the number of 
critical points of a reduced defining equation of $D^2_1(f)$ which move off the zero level as $t$ moves off $0$. 
The unfolding 
$$F(t_1,t_2,t_3,x,y,z)=(t_1,t_2,t_3,x,y^2+xz+x^2y,yz+t_1y+t_2z,z^2+y^3+t_3y)$$ is stable, by Mather's algorithm for the construction of stable germs as unfoldings of germs of rank 0, and $D^2_1(f_t)$ is the fibre of $D^2_1(F)$ over $t\in\CC^3$.  Let $G$ be an equation 
for $D^2_1(F)$,  let $g_t$ be its restriction to $\{t\}\times\CC^3$, and let $J^{\text{rel}}_G$ be the relative jacobian ideal 
$\left(\p G/\p X, \p G/\p Y, \p G/\p Z\right)$.  As in Subsection \ref{mui}, we compute the number of 
critical points of a reduced defining equation of $D^2_1(f)$ which move off the zero level as $t$ moves off $0$, as the 
intersection multiplicity
\beq\label{intmult}
\left(V(J^\text{rel}_G:G^\infty)\cdot (\{0\}\times\CC^3)\right)_{(0,0)}.
\eeq
In fact calculation shows that in this case $(J^\text{rel}_G:G^\infty)$ is equal to the transporter $(J^\text{rel}_G:G)$. However, unlike the situation discussed in Subsection \ref{mui}, $V(J^\text{rel}_G:G^\infty)$ is not Cohen-Macaulay; it has projective dimension 5 as $\OO_{\CC^6,0}$-module, while its codimension is 3. To
 compute the intersection multiplicity, we have to use Serre's {\it formule clef}, from \cite{serre}.
Denote $(J^\text{rel}_G:G^\infty)$ by $Q$; then
$$(V(Q),\{0\}\times\CC^3)_0=\sum_j(-1)^j\dim_\CC\text{Tor}^{\OO}_j\left(\frac{\OO}{Q},\frac{\OO}{(t_1,t_2,t_3)}\right)$$
where $\OO=\OO_{\CC^6,0}$. Since $t_1,t_2,t_3$ is a regular sequence there are at most three non-vanishing $\text{Tor}$ modules, for $j=0,1,2$. A {\it Macaulay} calculation shows that they have dimension $29, 3, 1$ respectively, so that 
 \beq\label{sier}\dim_\CC H_2(D^2_1(f_t))=(V(Q),\{0\}\times\CC^3)_{(0,0)}=29-3+1=27.\eeq
\end{proof}
It is striking that in this case $V(Q)$ is not Cohen-Macaulay. In all of the examples I know, where one uses the
procedure of Subsection \ref{mui} to calculate $\mu_I$, and $G$ is the defining equation of the image of the stable unfolding $F$, the corresponding space $V(J^{\text{rel}}_G:G^\infty)$ is Cohen Macaulay. 
 
To relate the homology of $D^2_1(f_t)$ to the homology of $D^2(f_t)$, we use the image computing spectral sequence: $D^2_1(f_t)$ is the image of the projection $\pi^2:D^2(f_t)\to U_t$. 
Taking account of the facts that $f_t$ has no quadruple points, so that $\pi^2$ has no triple points,  and that $H_1(D^2(f_t))=0$, the $E^1$ term is reduced to 
\beq\label{icss}\xymatrix{0=H_0^\alt(D^2(\pi^2))\ar[d]&H_1^\alt(D^2(\pi^2))\ar[d]_{d_1}&0\\
H_0(D^2(f_t))&0&H_2(D^2(f_t))}
\eeq
and the spectral sequence collapses here. So 
\beq\label{2eq}27=\dim_\CC H_2(D^2_1(f_t))=\dim_\CC H_2(D^2(f_t))+\dim_\CC H_1^\alt(D^2(\pi^2)).\eeq
Recall from Remark \ref{inter}  the isomorphism $i:D^3(f_t)\to D^2(\pi^2:D^2(f_t)\to U_t)$, given by
$(a,b,c)\mapsto ((a,b),(a,c))$. The involution on $D^2(\pi^2)$ lifts to the transposition $(2,3)$ on $D^3(f_t)$. Thus
under the induced isomorphism of first homology, $H_1^\alt(D^2(\pi^2))$ corresponds to the $-1$ eigenspace of $(2,3)_*$ on
$H_1(D^3(f_t))$.
On each copy of the 2-dimensional irreducible representation $\rho$, and on each copy of the sign representation, $(2,3)$ has 1-dimensional $-1$ eigenspace. Thus, using \ref{dimde} for the second equality,
\beq\label{it} \dim_{\CC}H_1^\alt(D^2(\pi^2))=\dim_{\CC}H_1^\alt(D^3(f_t))+\frac{1}{2}\dim_{\CC}H_1^\rho(D^3(f_t))=17,\eeq
and, by \eqref{2eq}, 
\beq\label{3eq}\dim_{\CC}H_2(D^2(f_t))=10.\eeq

\section{Calculations for the germ $(\CC^5,0)\to (\CC^6,0)$.}\label{5to6}

The germ 
$$f_0(x,y,a,b,c)=(x^2+ax+by,xy,y^2+cx+ay,a,b,c)$$
has $\mu_I=\mathcal{A}_e$-codimension $= 1$,  and versal unfolding
$$F(x,y,a,b,c,t)=(x^2+ax+by,xy,y^2+cx+(a+t)y,a,b,c,t).$$
Let $\xymatrix{U_t\ar@{->>}[r]^{f_t}&X_t}$ be a stable perturbation of $f_0$, with contractible domain $U_t\subset\CC^5$.
By \eqref{hgm},
$$1=\text{rank}\,H_5(X_t)=\text{rank}\,H_4^\alt(D^2(f_t))+\text{rank}\, H_3^\alt(D^3(f_t)).$$
As in the previous section, we approach $D^3(f_t)$ via its projection to $U_t$, $D^3_1(f_t)$. As before, $D^3_1(f_t)$ is defined by the pull-back of the second Fitting ideal of $\OO_{\CC^5,0}$ considered as $\OO_{\CC^6,0}$-module.
The $\OO_{\CC^6\times\CC,(0,0)}$-module $F_*\left(\OO_{\CC^5\times\CC, (0,0)}\right)$ has presentation 
\beq\label{pres6}\begin{pmatrix} Y^2-XZ-abZ-bcY+atY&aY+cX+tY&aY+bZ\\
 aY+cX+tY&-Z-ac&Y-bc\\
 aY+bZ&Y-bc&-X-ab-bt
 \end{pmatrix}.
       \eeq
 so
 $$\fitt_2=(Z+ac,Y-bc, X+(a+t)b)$$
 and
 $$F^*\fitt_2=(y^2+y(a+t)+xc+ac, xy-bc, x^2+xa+yb+ab+bt)$$
 $$=\text{min}_2\begin{pmatrix} -y&  -c   \\
      x+a &-y-a-t \\
       b &  x 
       \end{pmatrix}
       $$  
The corresponding  ideal for $t=0$ defines the 3-fold singularity $D^3_1(f_0)$. A {\it Macaulay2} calculation shows that the $T^1$ of $D^3_1(f_0))$ has dimension $1$. 
Therefore $D^3_1(f_0)$ is isomorphic to the unique non-ICIS codimension 2 Cohen Macaulay 3-fold singularity with $\tau=1$, which one can find in the table on page  22 of \cite{FruhbisZach}. This table also lists the Betti numbers of 
a smoothing, from which we obtain 
\beq\label{n-1}h_0(D^3_1(f_t))=1,\quad h_1(D^3_1(f_t))=0,\quad h_2(D^3_1(f_t))=1,\quad h_3(D^3_1(f_t))=0.\eeq 
In particular, $\chi(D^3_1(f_t))=2.$

Now
$D^3(f_t)$ and $D^3_1(f_t)$ are smoothings of $D^3(f_0)$  and $D^3_1(f_0)$.  Let $\pi:D^3(f_t)\to D^3_1(f_t)$ be the projection $\pi(P,Q,R)=P$.  Then $D^3(f_t)$ is a branched double cover of $D^3_1(f_t)$: for a generic point $P\in D^3_1(f_t)$, which shares its $f_t$-image with $Q$ and $R$, $\pi^{-1}(P)=\{(P,Q,R),(P,R,Q)\}$. Because there are no quadruple points, the branching is of two types:
\begin{itemize}
\y
over a point $P$ where $f_t$ has a stable singularity  of type $\Sigma^{1,1,0}$,  $\pi^{-1}(P)=\{(P,P,P)\}$.  The set of all such points $P$ is denoted
$\Sigma^{1,1}f_t$
\y
if $f_t(P)=f_t(Q)$ with $f_t$ an immersion at $P$, and of type $\Sigma^{1,0}$ at $Q$, then 
$\pi^{-1}(P)=\{(P,Q,Q)\}$, so $(P,Q,Q)$ is a branch point.  We denote the set of all such points $P$ by $D^{3}_{1,0}(f_t)$. Note that because  $\quad\pi^{-1}(Q)=\{(Q,P,Q),(Q,Q,P)\},$
$(Q,P,Q)$ and $(Q,Q,P)$ are not branch points.

\end{itemize}
Thus 
\beq\label{n0}\chi(D^3(f_t))=2\chi(D^3_1(f_t))-\chi(D^{3}_{1,0}(f_t))=4-\chi(D^{3}_{1,0}(f_t)).\eeq
\subsection{Equations for $\Sigma^{1,1}f$}
The ramification ideal $R_f$, generated by the $5\times 5$ minors of the jacobian matrix $J$ of $f_0$ defines the non-immersive locus $\Sigma f$ of $f_0$. Then
$\Sigma^{1,1}(f)$ is defined by the ideal of maximal minors of the matrix obtained by concatenating $J$ with the jacobian matrix of a set of generators of $R_f$ (see e.g.\cite{matherTB}). By removing from this ideal an $\mathfrak{m}$-primary component we obtain the ideal  
$$S:=(3y + a, 3x + a, ac - 3bc, ab - 3bc, a ^2- 9bc),$$
easily recognised as defining a curve isomorphic to the germ of the union of the three coordinate axes in $(\CC^3,0)$. This has 
$\delta=2$ and therefore $\mu=2\delta-r+1=2$. It is not quite evident that this is deformed flat in a deformation of 
$f_0$, but nevertheless this is the case. The corresponding locus for the 1-parameter versal deformation $F$ of $f_0$ has an
$\mathfrak{m}$-primary component, whose removal leaves a 2-dimensional Cohen-Macaulay component which 
restricts to $\Sigma^{1,1}f_0$.
\vs
\subsection{Equations for $D^{3}_{1,0}(f_0)$}\label{shadow}
By the description above, $D^{3}_{1,0}(f_0)$ is the ``shadow component'' of $f_0^{-1}(f(\Sigma f_0))$, that is, the closure of
$f_0^{-1}(f_0(\Sigma f_0))\ssm \Sigma f_0$. To find equations for it, we first look for equations for the support of $f_{0*}(\OO_S/{\mathcal R}_{f_0})$. Let $I_0$ be the radical of the zero'th Fitting ideal of $f_{0*}(\OO_S/{\mathcal R}_{f_0})$, let $I_1=f_0^*(I_0)$, and let $I_2$ be the saturation $I_1:{\mathcal R}_{f_0}^\infty$, in this case equal to  $I_1:{\mathcal R}_{f_0}^2$. After some effort one finds that 
$I_2$ is the ideal of maximal minors of the $2\times 4$ matrix  
$$\pmat a  &     b  & x   &   y  \\ 
       -3y+a& x+a &-y-a& 3y-a+4c\emat.
       $$ 
This is isomorphic to the cone over the rational normal curve of degree 4 (Pinkham's example). In the versal deformation $F$, the same construction leads to the ideal of maximal minors of the $2\times 4$ matrix 
$$
\pmat a  &     b  & x   &   y+t   \\ 
       -3y+a+t& x+a &-y-a-t& 3y-a+4c-t\emat
$$ 
One checks that this defines a smoothing of $D^3_{1,0}(f_0)$, over the Artin component of the base space (since it is given by the minors of a $2\times 4$ matrix). So the only non-zero reduced Betti number is $\beta_2=1$ ((see e.g.
\cite{nemethi04} page 173). In particular 
\beq\label{n1}\chi(D^3_{1,0}(f_t))=2.\eeq

              \subsection{Homology of $D^3(f_t)$}\label{D35to6}
By \eqref{hgm} and Lemma \ref{conl}, 
\beq\label{rest}H_i(D^3(f_t))=H^\alt_i\oplus H^\rho_i.\eeq
Denote by $h_i^\alt$ and $h_i^\rho$ the ranks of these summands. 

Because there are no quadruple points, $D^3_1(f_t)$ is the quotient of $D^3(f_t)$ by the $\ZZ_2$-action generated by the transposition $(2,3)(P,Q,R)=(P,R,Q)$.  So 
$H_i(D^3_1(f_t))$ is the part of $H_i(D^3(f_t))$ invariant under $(2,3)_*$. Since $H_i^T(D^3(f_t))=0$ for $i>0$, 
the $(2,3)_*$ -invariant part of $H_i(D^3(f_t))$ is $(2,3)_*$-invariant part of $H^\rho_i(D^3(f_t))$, and thus isomorphic to the sum of copies of the subspace of $\rho$ invariant under
$(2,3)$. The $(2,3)$-invariant subspace of $\rho$ is 1-dimensional. Thus, 
\beq\label{n1} h_i(D^3_1(f_t))=\frac{1}{2}h_i^\rho(D^3(f_t))\eeq
for $i>1.$
Hence, by \eqref{n-1}, 
\beq\label{n2}h_1^\rho(D^3(f_t))=0,\  \ h_2^\rho(D^3(f_t))=2,\ \  h_3^\rho(D^3(f_t))=0.\eeq
On the other hand, as $D^3(f_t)$ is a branched cover of degree 2 of $D^3_1(f_t)$, branched along $D^3_{1,0}(f_t)$, it follows that 
$$\chi(D^3(f_t))=2\chi(D^3_1(f_t))-\chi(D^{3}_{1,0}(f_t))=2.$$
Putting this together with \eqref{n1}, we have
$$2=\chi(D^3(f_t))=1-\left(h_1^\rho+h_1^\alt\right)+\left(h_2^\rho+h_2^\alt\right)-\left(h_3^\rho+h_3^\alt\right)=1-h^\alt_1+h_2^\alt+2-h^\alt_3.$$
so
\beq\label{n3}-1=-h_1^\alt+h_2^\alt-h_3^\alt.\eeq
By \cite[Theorem 4.6]{Hou97}, the alternating homology of the multiple point spaces of a stable perturbation of a finitely determined map-germ is concentrated in middle dimension.  As $D^3(f_t)$ is a 3-fold, this means
$h_i^\alt(D^3(f_t))=0$  for $i\neq 3$ and so from \eqref{n3}, $h_3^\alt(D^3_1(f_t))=1$. Since here $\mu_I=1$, we conclude from \eqref{hgm} that $h_4^\alt(D^2(f_t))=0$. Also, from \eqref{n2} and \eqref{rest}, we conclude that $H_1(D^3(f_t))=0$,
$\dim_{\QQ}H_2(D^3(f_t))=2$ and $\dim H_3(D^3(f_t))=1$.

\begin{remark}{\em An application of the extended version of the Lefschetz Fixed Point Theorem gives the same conclusion:  the fixed set of the involution 
$(2,3)$ on $D^3(f_t)$ is homeomorphic to the branch locus $D^3_{1,0}(f_t))$, and so by the extended version of the Lefschetz Fixed Point Theorem, 
$$2=\chi(\text{Fix}(2,3))=\sum_i(-1)^i\left(\text{Tr}(2,3)_*:H_i(D^3(f_t))\to H_i(D^3(f_t))\right)=1+h_1^\alt-h_2^\alt+h_3^\alt$$
(recall that $\chi_\rho(2,3)=0$).
This equation is of course the same as \eqref{n3}.}
\end{remark}
\subsection{Double Points}\label{defdub}
The double locus of $F$ is defined by $F^*(\fitt_2(F_*(\OO_S))$, which is a principle ideal generated by the composite with $F$ of the determinant of the lower right $2\times 2$ submatrix of \eqref{pres6}. We will call this composite $G$. As in Lemma \ref{nocm},  $D^2_1(f_t)$ has the homotopy type of a wedge of $4$-spheres, whose number is the intersection number of $V(J^{\text{rel}}_G:G^\infty)$ with $\CC^5\times \{0\}$.  A {\it Macaulay2} calculation gives
\beq\label{d2}
 (J^{\text{rel}}_G:G^\infty)=(J^{\text{rel}}_G:G)=(3c - t, 3b + t, 2a + t, 6y + t, 6x - t).
 \eeq
 The zero-locus of this ideal is a smooth curve of degree 1 over the $t$-axis, so $D^2_1(f_t)$ is homotopy-equivalent to a single $4$-sphere, by
 Siersma's theorem \cite{siersma}. 
 \def\ftt{f^{(2)}_t}
  
The map of pairs $\xymatrix{(D^2(f_t),D^3_2(f_t))\ar[r]^{\pi^2_1}& (D^2_1(f_t),D^3_1(f_t))}$ induces a morphism between the long exact sequences of reduced homology of the pairs. Because $\pi^2_1:D^2\to D^2_1$ is an isomorphism outside $D^3_2$, the morphisms $H_i(D^2, D^3_2)\simeq H_i(D^2_1,D^3_1)$ are isomorphisms for all $i$. 
From the segment
$$\xymatrix{\cdots \ar[r]&H_2(D^2,D^3_2)\ar@{=}[d]\ar[r]&H_1(D^3_2)\ar[d]\ar[r]&H_1(D^2)\ar[d]\ar[r]&H_1(D^2,D^3_2)\ar@{=}[d]\ar[r]&0\\
\cdots \ar[r]&H_2(D^2_1,D^3_1)\ar[r]&H_1(D^3_1)\ar[r]&H_1(D^2_1)\ar[r]&H_1(D^2_1,D^3_1)\ar[r]&0}
$$
and the fact that $H_1(D^2_1)=0=H_1(D^3_2)$ (the latter equality because $D^3_2\simeq D^3$, as there are no quadruple points) we see that
$H_1(D^2(f_t))=0.$ Because $H_2(D^2_1,D^3_1)$ is sandwiched between $0$'s in the lower sequence,  continuing to the left we have 
$$\xymatrix{&\QQ\ar@{=}[d]&&&\QQ^2\ar@{=}[d]\\
H_4(D^2,D^3_2)\ar@{=}[d]\ar[r]&H_3(D^3_2)\ar[r]\ar[d]&H_3(D^2)\ar[d] \ar[r]&H_3(D^2,D^3_2)\ar@{=}[d]\ar[r]&H_2(D^3_2)\ar[d]\ar[r]&H_2(D^2)\ar[d]\ar[r]&0\\
H_4(D^2_1,D^3_1)\ar[r]\ar@{=}[d]&H_3(D^3_1)\ar@{=}[d]\ar[r]&H_3(D^2_1)\ar@{=}[d]\ar[r]&H_3(D^2_1,D^3_1)\ar[r]\ar@{=}[d]&H_2(D^3_1)\ar[r]\ar@{=}[d]&H_2(D^2_1)\ar[r]\ar@{=}[d]&0\\
\QQ&0&0&\QQ&\QQ&0}
$$
We deduce successively
\begin{itemize}
\y $H_3(D^2,D^3_2)\simeq\QQ$ and $H_3(D^2,D^3_2)\to H_2(D^3_2)$ is injective, and therefore
\y
$H_2(D^2)\simeq\QQ$.
\y
$H_3(D^3_2)\to H_3(D^2)$ is surjective
\end{itemize}
The left end of the upper sequence is thus
$$\xymatrix{0\ar[r]&H_4(D^2)\ar[r]&H_4(D^2,D^3_2)\ar[r]&H_3(D^3_2)\ar[r]&H_3(D^2)\ar[r]&0}$$
with the two inner modules both isomorphic to $\QQ$.
 \subsection{Homology of $M_2(f_t)$}
 By comparing the homology of $D^2(f_t)$ and $M_2(f_t)$ (which we will shortly determine), we might hope to gain some information about the homology of $D^2(f_t)$. 
 All of the homology groups of $M_2(f_t)$ vanish. This can be seen with the help of the morphism $f_{t*}$ from the long exact sequence of the pair $(U_t,D^2_1(f_t))$ to the long exact sequence of the pair $(X_t, M_2(f_t))$. Because $f_t$ is an isomorphism outside $D^2_1(f_t)$, the morphisms of relative homology groups
 $$f_{t*}:H_i(U_t,D^2_1(f_t))\to H_i(X_t,M_2(f_t))$$ are all isomorphisms. From the top row of the diagram
 $$\xymatrix{0\ar[r]&0=H_5(U_t)\ar[r]\ar[d]&H_5(U_t,D^2_1)\ar[r]\ar[d]&H_4(D^2_1)\ar[r]\ar[d]&H_4(U_t)=0\ar[d]\\
 0\ar[r]&H_5(X_t)\ar[r]&H_5(X_t,M_2)\ar[r]&H_4(M_2)\ar[r]&H_4(X_t)=0}
 $$
 we see that $H_5(U_t,D^2_1)\simeq\QQ$. Hence $H_5(X_t,M_2)\simeq \QQ$ also, and then from the bottom row it follows that $H_4(M_2)=0.$ A similar argument shows that $H_i(M_2)=0$ for $0<i<4$. 

In fact Houston shows in \cite{Hou97} by a rather more sophisticated argument that for a stable perturbation $f_t$ of an $\mathcal A$-finite germ $f_t$, all of the $M_k(f_t)$ are wedges of spheres in middle dimension.

 \subsection{Relation between $D^2$ and $M_2$}
 There is a surjective map $f_t^{(2)}:D^2(f_t)\to M_2(f_t)$, $f^{(2)}_t(P,Q)=f_t(P)$. The multiple point spaces of $f^{(2)}_t$ are related to those of $f_t$, but are not identical. Consider the following maps:
 $$\alpha:D^2(f)\to D^2(\ftt),\quad (P,Q)\mapsto (\bigl((P,Q),(Q,P)\bigr)$$
 $$\beta:D^2(\ftt)\to D^2(f),\quad \bigl((P,Q),(R,S)\bigr)\mapsto (P,R).$$
 Denote by $(1,2)$ the usual involution on $D^2$. The diagrams 
 $$\xymatrix{D^2(f_t)\ar[d]_\a\ar[r]^{(1,2)}&D^2(f_t)\ar[d]^\a\\
 D^2(\ftt)\ar[r]^{(1,2)}&D^2(\ftt)}\quad\quad \xymatrix{D^2(f)\ar[r]^{(1,2)}&D^2(f)\\
 D^2(\ftt)\ar[r]^{(1,2)}\ar[u]^\beta&D^2(\ftt)\ar[u]^\beta}
 $$
 both commute, and $\beta\circ\alpha$ is the identity on $D^2(f)$. It follows that $\a$ and $\beta$ induce morphisms $$\xymatrix{H_i^\alt(D^2(f_t))\ar@/^/[r]^{\a_*}& H_i^\alt(D^2(\ftt)\ar@/^/[l]^{\beta_*}}$$
 and $\beta_*\circ\alpha_*$ is the identity.
 
 However $\a$ is not surjective and $\beta$ is not injective. Suppose that $(P,Q,R)\in D^3(f_t)$ with $P, Q, R$ pairwise distinct. Then
 $$\bigl((P,Q),(Q,P)\bigr), \bigl((P,Q),(Q,R)\bigr), \bigl((P,R),(Q,P)\bigr), \bigl((P,R),(Q,R)\bigr)$$
 all lie in $D^2(\ftt)$ and all are mapped by $\beta$ to $(P,Q)$. And, of these, only $\bigl((P,Q),(Q,P)\bigr)$ is in the 
 image of $\a$.
 
 We draw no further conclusion from this, but ask whether further consideration of the multiple point spaces
 of the map $f_t^{(2)}$ and indeed of $f^{(k)}_t$ for higher $k$ may provide useful information.
 
 Nevertheless, from the vanishing of $H_1(M_2(f_t))$, and the image-computing spectral sequence, we obtain a second argument that $H_1(D^2(f_t))=0$. 

\def\cprime{$'$} \def\cprime{$'$} \def\cprime{$'$} \def\cprime{$'$}
  \def\cprime{$'$} \def\cprime{$'$}
\providecommand{\bysame}{\leavevmode\hbox to3em{\hrulefill}\thinspace}
\providecommand{\MR}{\relax\ifhmode\unskip\space\fi MR }
\providecommand{\MRhref}[2]{%
  \href{http://www.ams.org/mathscinet-getitem?mr=#1}{#2}
}
\providecommand{\href}[2]{#2}


\begin{thebibliography}{MNBPS12}

\bibitem[AM13]{am}
Ay{\c{s}}e Alt{\i}nta{\c{s}} and David Mond, \emph{Free resolutions for
  multiple point spaces}, Geom. Dedicata \textbf{162} (2013), 177--190.
  \MR{3009540}

\bibitem[dJ90]{dJ90}
Theo de~Jong, \emph{The virtual number of {$D_\infty$} points. {I}}, Topology
  \textbf{29} (1990), no.~2, 175--184. \MR{1056268 (91f:32043)}

\bibitem[dJvS91]{dJvS}
T.~de~Jong and D.~van Straten, \emph{Disentanglements}, Singularity theory and
  its applications, {P}art {I} ({C}oventry, 1988/1989), Lecture Notes in Math.,
  vol. 1462, Springer, Berlin, 1991, pp.~199--211. \MR{1129033 (93a:14003)}

\bibitem[FKZ15]{FruhbisZach}
A.~Fruehbis-Krueger and M.~Zach, \emph{On the vanishing topology of isolated
  {C}ohen-{M}acaulay codimension 2 singularities}, arXiv:1501.01915 (2015).

\bibitem[GM93]{gm}
Victor~V. Goryunov and David Mond, \emph{Vanishing cohomology of singularities
  of mappings}, Compositio Math. \textbf{89} (1993), no.~1, 45--80. \MR{1248891
  (94k:32058)}

\bibitem[Gor95]{gor}
Victor~V. Goryunov, \emph{Semi-simplicial resolutions and homology of images
  and discriminants of mappings}, Proc. London Math. Soc. (3) \textbf{70}
  (1995), no.~2, 363--385. \MR{1309234 (95j:32050)}

\bibitem[GS83]{GrSt81}
Gert-Martin Greuel and Joseph Steenbrink, \emph{On the topology of smoothable
  singularities}, Singularities, {P}art 1 ({A}rcata, {C}alif., 1981), Proc.
  Sympos. Pure Math., vol.~40, Amer. Math. Soc., Providence, R.I., 1983,
  pp.~535--545. \MR{713090}

\bibitem[Hou97]{Hou97}
Kevin Houston, \emph{Local topology of images of finite complex analytic maps},
  Topology \textbf{36} (1997), no.~5, 1077--1121. \MR{1445555 (98g:32064)}

\bibitem[Kle81]{kmulti1}
Steven~L. Kleiman, \emph{Multiple-point formulas. {I}. {I}teration}, Acta Math.
  \textbf{147} (1981), no.~1-2, 13--49. \MR{631086 (83j:14006)}

\bibitem[Mat73]{matherTB}
John~N. Mather, \emph{On {T}hom-{B}oardman singularities}, Dynamical systems
  ({P}roc. {S}ympos., {U}niv. {B}ahia, {S}alvador, 1971), Academic Press, New
  York, 1973, pp.~233--248. \MR{0353359 (50 \#5843)}

\bibitem[Mil68]{Milnor}
J.~Milnor, \emph{Singular points of complex hypersurfaces}, Ann. of Math.
  Studies, vol.~61, Princeton University Press, Princeton, 1968.

\bibitem[MM89]{marar-mond}
Washington~Luiz Marar and David Mond, \emph{Multiple point schemes for corank
  {$1$} maps}, J. London Math. Soc. (2) \textbf{39} (1989), no.~3, 553--567.
  \MR{1002466 (91c:58010)}

\bibitem[MM96]{grp}
\bysame, \emph{Real map-germs with good perturbations}, Topology \textbf{35}
  (1996), no.~1, 157--165. \MR{1367279}

\bibitem[MNB08]{mn}
W.~L. Marar and J.~J. Nu{\~n}o-Ballesteros, \emph{A note on finite determinacy
  for corank 2 map germs from surfaces to 3-space}, Math. Proc. Cambridge
  Philos. Soc. \textbf{145} (2008), no.~1, 153--163. \MR{2431646}

\bibitem[MNBPS12]{mnp}
W.~L. Marar, J.~J. Nu{\~n}o-Ballesteros, and G.~Pe{\~n}afort-Sanchis,
  \emph{Double point curves for corank 2 map germs from {$\Bbb C^2$} to {$\Bbb
  C^3$}}, Topology Appl. \textbf{159} (2012), no.~2, 526--536. \MR{2868913}

\bibitem[Mon87]{Mon87}
David Mond, \emph{Some remarks on the geometry and classification of germs of
  maps from surfaces to {$3$}-space}, Topology \textbf{26} (1987), no.~3,
  361--383. \MR{899055}

\bibitem[Mon15]{mond15}
\bysame, \emph{Some open problems in the theory of singularities of mappings},
  J. Singul. \textbf{12} (2015), 141--155. \MR{3317146}

\bibitem[MP89]{MP89}
David Mond and Ruud Pellikaan, \emph{Fitting ideals and multiple points of
  analytic mappings}, Algebraic geometry and complex analysis ({P}\'atzcuaro,
  1987), Lecture Notes in Math., vol. 1414, Springer, Berlin, 1989,
  pp.~107--161. \MR{1042359 (91e:32035)}

\bibitem[NBPS15]{NuPe15}
J.J. Nu{\~n}o-Ballesteros and G.~Pe{\~n}afort-Sanchis, \emph{On multiple point
  schemes}, arXiv:1509.04990 (2015).

\bibitem[N{\'e}m04]{nemethi04}
Andr{\'a}s N{\'e}methi, \emph{Invariants of normal surface singularities}, Real
  and complex singularities, Contemp. Math., vol. 354, Amer. Math. Soc.,
  Providence, RI, 2004, pp.~161--208. \MR{2087811}

\bibitem[Ser57]{serre}
Jean-Pierre Serre, \emph{Alg{\`e}bre locale. {M}ultiplicit{\'e}}, Lecture Notes
  in Mathematics, Vol. 11, Springer-Verlag, Berlin, 1957.

\bibitem[Sha14]{sharland14}
Ay{\c{s}}e~Alt{\i}nta{\c{s}} Sharland, \emph{Examples of finitely determined
  map-germs of corank 2 from {$n$}-space to {$(n+1)$}-space}, Internat. J.
  Math. \textbf{25} (2014), no.~5, 1450044, 17. \MR{3215220}

\bibitem[Sie91]{siersma}
Dirk Siersma, \emph{Vanishing cycles and special fibres}, Singularity theory
  and its applications, {P}art {I} ({C}oventry, 1988/1989), Lecture Notes in
  Math., vol. 1462, Springer, Berlin, 1991, pp.~292--301. \MR{1129039
  (92j:32129)}

\end{thebibliography}
\end{document}